\documentclass[a4paper,fleqn]{cas-sc}
\usepackage{enumitem}
\usepackage{booktabs}
\usepackage{subcaption}
\usepackage{float}
\usepackage{amsmath}
\usepackage{amsthm}
\usepackage[authoryear,longnamesfirst]{natbib}

\def\tsc#1{\csdef{#1}{\textsc{\lowercase{#1}}\xspace}}
\tsc{WGM}
\tsc{QE}

\newtheorem{theorem}{Theorem}
\newtheorem{lemma}[theorem]{Lemma}

\newdefinition{definition}{Definition}
\newdefinition{rmk}{Remark}
\newdefinition{example}{Example}
\newproof{pf}{Proof}

\begin{document}
\let\WriteBookmarks\relax
\def\floatpagepagefraction{1}
\def\textpagefraction{.001}

\shorttitle{Static MFG for Wind Farm Siting}    

\shortauthors{}  

\title [mode = title]{A Static Mean Field Game for Optimal Renewable Energy Investment}  

\tnotemark[1] 


\tnotetext[1]{This document is the result of the research project partially funded by the Transdisciplinary Scholarship Connector Grants University of Calgary and Sezer's NSERC Discovery Grant RGPIN-2025-05373.}

%

\author[1]{Deniz Sezer}



\ead{adsezer@ucalgary.ca}

\ead[url]{}

\credit{Conceptualization, Methodology, Supervision, Writing – review and editing, Project administration, Funding acquisition}

\affiliation[1]{organization={Department of Mathematics and Statistics, University of Calgary},
            addressline={2500 University Dr NW}, 
            city={Calgary},
            postcode={T2N 1N4}, 
            state={Alberta},
            country={Canada}}

\author[1]{Shanukie Vithana}

\cormark[1]


\ead{shanukie.vithana@ucalgary.ca}


\credit{Data curation, Methodology, Software, Validation, Formal analysis, Investigation, Writing – original draft, Writing – review and editing}

\affiliation[1]{organization={Department of Mathematics and Statistics, University of Calgary},
            addressline={2500 University Dr NW}, 
            city={Calgary},
            postcode={T2N 1N4}, 
            state={Alberta},
            country={Canada}}

\author[2]{Sara Hastings-Simon}


\ead{sara.hastingssimon@ucalgary.ca}


\credit{Conceptualization, Validation, Investigation, Resources, Writing – review and editing }

\affiliation[2]{organization={Department of Earth, Energy, and Environment, University of Calgary},
            addressline={2500 University Dr NW}, 
            city={Calgary},
            postcode={T2N 1N4}, 
            state={Alberta},
            country={Canada}}

\author[1]{Souren Iwazian,}



\ead{souren.iwazian@ucalgary.ca}


\credit{Data curation, Methodology, Software, Formal analysis, Investigation, Validation, Writing - Original Draft}

\affiliation[1]{organization={Department of Mathematics and Statistics, University of Calgary},
            addressline={2500 University Dr NW}, 
            city={Calgary},
            postcode={T2N 1N4}, 
            state={Alberta},
            country={Canada}}




\begin{abstract}
Mean Field Game (MFG) theory provides a mathematical framework for modeling the strategic behavior of a large number of interacting agents whose individual decisions are influenced by the aggregate behavior of the population. This framework is well-suited for studying decentralized investment decisions in competitive electricity markets, where the actions of individual wind farm developers collectively shape market outcomes.
We develop a static MFG formulation for optimal wind energy capacity siting in Alberta, Canada. The revenue model for agents is based on the expected wind resource at a location and the covariance of this wind resource with other locations. Agents choose locations to maximize long-run revenue while accounting for spatial correlations in wind resource availability. Spatial wind dependence is incorporated through an empirically calibrated covariance structure, combining a PCA-based component estimated from historical weather station data with a parametric residual kernel, to capture variability across locations.
The equilibrium investment problem can be formulated as a quadratic program, which we solve for four policy scenarios that progressively restrict the feasible siting area, incorporating viewscape and transmission constraints beyond a baseline of minimal siting restrictions. These policy and infrastructure-driven land use restrictions are represented as constraints on the agents' action space, allowing us to examine their economic implications through a comparison of the MFG equilibria across policy scenarios. Our findings highlight the trade-off between regulatory land-use objectives and economic efficiency in the use of renewable energy resources in the short and long term, and provide quantitative insights into how policy design shapes the spatial distribution and financial performance of wind investments. The static MFG formulation could be used to inform electricity system planning including, transmission planning and renewable energy development planning.

\end{abstract}


\begin{highlights}
\item A novel spatial mean field game framework for wind farm siting under policy constraints.
\item Spatial diversification for maximizing revenue via a covariance penalty.
\item Estimate revenue at any location in Alberta based solely on wind resource.
\item Optimizing wind farm distribution and guiding future transmission and renewable planning.
\end{highlights}


\maketitle

\section{Introduction}\label{}

As the global community strives to achieve net-zero emissions by 2050, a goal set by the Paris Agreement, the energy sector faces a critical transformation. In the electricity sector, this transition requires a significant shift from fossil fuel generation to lower-emitting sources, such as wind and solar power. However, these weather-dependent resources are inherently variable in time and space, introducing a different type of complexity in their planning and operation. The optimal location for construction of wind and solar infrastructure in an electricity grid depends not only on the renewable potential at a location but also the correlation of energy generation at other sites, particularly as the share of variable renewable energy generation grows. This is the Merit Order Effect which is the principle used in Alberta's electricity market which prioritizes energy sources based on their cost. It arranges energy sources in ascending order of their supply offers. Typically renewable energy producers together with co-generators bid price zero and can be found at the bottom of the curve and the rest of the conventional producers are stacked in ascending order of their offer prices. The system operator dispatches the low marginal cost generators first to meet the exogenous demand and the residual demand (defined as the difference between total electricity demand and aggregate renewable generation) is met by the conventional producers. This ensures the most economical use of available energy resources. In Alberta’s deregulated electricity market, all generators are paid the spot price for the power they supply. The market clearing price or the spot price is the offer price of the most expensive generator that is needed to meet the demand and all generators that were dispatched are paid this price. 

Therefore, a location with high average wind speed may not always be ideal if it is highly correlated with other sites, since simultaneous high production across sites reduces the value of the generation to the grid. In competitive markets this is seen through suppression of electricity prices and diminish revenues. On the other hand, a site with moderate wind potential but a lower correlation may offer higher value and thus higher revenues. The siting of new wind and solar infrastructure can be further constrained by the availability of existing transmission system interconnections or the ability to build new transmission infrastructure.

In Alberta, recent policy developments limit the options for renewable energy site selection by removing significant areas from potential renewable energy development. The provincial government has introduced regulations restricting renewable energy projects on Class I and II (and in some cases Class III) agricultural lands and in pristine viewscapes. A full list of imposed restrictions is available in the official report by \citet{AlbertaRestrictions}. These policy changes may lead to suboptimal wind and solar farm distribution, which could lead to more expensive energy production, reduced renewable energy development, and different distribution of benefits from development. The magnitude of these impacts depends on the overlap between the restricted areas and the optimal areas for development. To evaluate the impacts, we propose a spatial Mean Field Game (MFG) framework to optimize wind farm siting with and without restrictions. Each agent in our model chooses a location to build a wind farm based on its expected revenue, which depends on both local wind resources and spatial correlations with other farms. This framework enables us to study equilibrium outcomes and assess the implications of land-use policies, ultimately informing more sustainable and economically viable energy planning throughout the province.

In classical MFGs, agent dynamics are typically modeled as diffusion processes over time, allowing for tractable equilibrium analysis. However, in problems like wind farm siting, it is more natural to model the environment using spatial processes that capture the geographic structure and correlation of renewable resources across locations. Existing work on spatial statistics, such as the low-rank approximations and hierarchical models developed by \citet{Wikle2019} and the foundational treatment in \citet{Cressie2011} provides the tools to characterize these dependencies. Yet, there remain gaps in incorporating spatial covariance structures directly into the strategic reward functions of MFGs. Our work addresses this by embedding a spatial covariance-informed reward in a static MFG framework, allowing us to study how geographic correlations influence equilibrium patterns of renewable energy investment. 

An important yet often overlooked aspect of renewable integration is the role of spatial diversification in smoothing wind energy production. While it is well established that geographic dispersion can reduce aggregate variability through the averaging effect, \citet{VariabilityGenerators}, its impact on revenue outcomes has received less attention. Recent work, such as \citet{SpatialDivProfit} examines spatial differentiation in wind turbine profitability but does not explicitly model how covariance structures between sites shape revenue risks. Our work fills this gap by incorporating a spatial covariance penalty into the reward function, enabling direct analysis of how clustering affects both production variability and the value of the generation expressed through financial returns. 

Several recent studies have made substantial contributions to spatial planning for renewable energy deployment. For instance, \citet{Elkadeem2025}, proposed a high-resolution spatio-temporal decision-making model (STDMM), to optimize the placement of PhotoVoltic (PV), wind, and hybrid systems in Saudi Arabia using ERA5 reanalysis data and over 20 spatial criteria. Their model integrates techno-economic potential and infrastructure costs, offering a comprehensive feasibility analysis for hybrid systems. However, their framework primarily centers on cost minimization and resource complementarity, with no explicit consideration of revenue maximization, market risks, or investor behavior.  Moreover, their multi-criteria decision-making (MCDM) approach relies on expert opinion and does not account for strategic investor actions or spatial generation correlations that might affect market outcomes.

\citet{Younes2023} present a geospatial-assisted multi-criteria decision-making (MCDM) model to assess the geographical-technical-economic (GTE) potential of solar photovoltaic (PV) and onshore wind turbine (WT) systems in Egypt. Their analysis employs 16 restrictive and evaluation criteria, spanning climatological, technical, economic, environmental, and social dimensions, to identify suitable sites at a high resolution of 1 $km^2$. The study provides detailed estimates of energy yields, demonstrating that optimal solar and wind sites could meet Egypt’s projected electricity demand by 2030. Additionally, it calculates the levelized cost of electricity (LCOE), for onshore wind, highlighting their cost competitiveness with conventional generation. While the work adds depth to national-scale siting analyses, it remains focused on static site suitability and LCOE, and does not consider inter-site wind output correlations or revenue implications in a deregulated market environment.

\citet{Soysal2025} analyzes the financial feasibility of wind power investments in Denmark under full market exposure, emphasizing how price volatility, carbon pricing, and merit-order effects influence the cost of capital. Using machine learning-based electricity price forecasts, the study demonstrates revenue erosion due to output-price correlations and identifies limits to carbon pricing as a driver of wind investments. However, the model treats location as exogenous and does not account for spatial heterogeneity or siting decisions.

Our work bridges these domains by embedding spatial covariance structures which govern both generation variability and market price exposure, directly into a strategic investment model. Unlike GIS-MCDM approaches, we model agents as optimizing revenue under endogenous interactions captured through a Mean Field Game (MFG) framework. And in contrast to market-based investment models, we treat location as a strategic choice variable, enabling analysis of equilibrium spatial patterns in the presence of land-use constraints and policy restrictions.

Building on these existing spatial planning approaches, this paper addresses the critical challenge of optimizing wind farm siting under spatial and policy constraints by developing a novel spatial Mean Field Game (MFG) framework that explicitly incorporates geographic covariance structures into investment decisions. Our approach advances existing methods in three key ways: First, we introduce a spatially informed reward function that captures how inter-site wind generation correlations affect revenue, filling a gap in renewable energy planning literature. Secondly, we empirically validate our model using Alberta wind data and power market data, achieving high explanatory power, $R^2$ = 75\% for expected capacity factor estimation, 90\% spatial variability captured through a hybrid PCA-parametric covariance approach and 84\% for the proposed revenue model. Third, we demonstrate how the resulting equilibrium solutions can guide policymakers and investors in balancing land use restrictions with optimal spatial diversification.  The significance of this analysis is that it enables us to estimate potential revenue at any location in Alberta based solely on wind resource characteristics. This can serve as a valuable tool for assessing the profitability of prospective wind farm sites, optimizing wind farm distribution, and guiding future transmission line development.

\section{\label{sec:level2} Static Mean Field Game Formulation}

\subsection{\label{sec:level2.1}Setup of the problem}

In a static game, all agents choose their actions simultaneously. We assume that the conventional supply is exogenous to the system. This simplifies the model by allowing us to focus on how renewable agents interact with each other and affect prices through residual demand. 

Consider $N$ investors who choose an action $(c, X)$ in a finite action space $A$.
$$A=\{c^1,...,c^k\}\times\{X^1,...,X^n\}.$$
\noindent
We assume that $\{c^1,...,c^k\}$ are uniformly spaced points that partition the interval $[\tilde{c},\tilde{C}]$, where $\tilde{c},\tilde{C}>0$ and $\{X^1,..,X^n\}$ are uniformly spaced grid points in a bounded domain in $\mathbb{R}^2$. The action $(c,X)$ means that an agent invests \textbf{$\frac{c}{N}$} units of capacity at the location $X$. Therefore, agent $i$'s chosen capacity $C_i=\frac{c_i}{N}$ is between a minimum of $\frac{\tilde{c}}{N}$ and maximum of $\frac{\tilde{C}}{N}$.  With this parameterization, the total capacity invested is in the interval  $[\tilde{c},\tilde{C}]$.
\begin{equation}\label{eq:Total_capacity_constraint}
   \tilde{c} \leq \sum_{i=1}^{N}C_i \leq \tilde{C}.
\end{equation}

To prevent unlimited investment in a single location, we introduce a density constraint as follows:
\begin{equation}\label{eq:density_constraint}
    \int_{\tilde{c}}^{\tilde{C}}c\mu^N(dc,X^j)\leq C_{max},
\end{equation}
where $\mu^N(dc,X_j)=\frac{1}{N}\sum_{i=1}^{N}\delta_{c_i}(dc).1_{\{X_i=X_j\}}$.  Indeed, if equation \eqref{eq:density_constraint} is true, and $(c_i,X_i)$ denotes the action of the $i^{th}$ investor, we have:
$$\sum_{i=1}^{N} \frac{c_i}{N} 1_{\{X_i=X^j\}} \leq C_{max},$$
for all $X^j\in \{X^1,...,X^n\}$. Hence $C_{max}$ is interpreted as the maximum total capacity that can be built at a single location.

Next, we formulate the revenue of each agent. Each agent $i$ maximizes a reward function $J_i^N : A^N\longrightarrow \mathbb{R}$, which assigns a payoff based on the agent's action and the actions of all the other agents. For the agent $i$ with the action $(c_i,X_i)$, the yearly average revenue proposed is:
\begin{equation}
    Revenue_i=\lambda \frac{c_i}{N}E[W(X_i)]-\sigma \sum_{j=1}^{N}\frac{c_i}{N}Cov(W(X_i),\frac{c_j}{N}W(X_j)),
\end{equation}

where $W(x)$ is the capacity factor(wind resource density) at location $x$. The capacity factor is defined as
\begin{equation}\label{eq:CapacityFactorEq}
    W(x) := \frac{\text{Power Generated(MW)}}{\text{Capacity(MW)}}.
\end{equation}

Equivalently, every agent optimizes 
\begin{equation}
    F((c_i,X_i),\mu^N)=\lambda {c_i} E[W(X_i)]-\sigma \sum_{j=1}^{N}{c_i}Cov(W(X_i),\frac{c_j}{N}W(X_j))
\end{equation}

We assume that $\lambda$ and $\sigma$ are independent of $\mu$, and we have the following interpretations. 

\textbf{Interpretation of $\lambda$:} $\lambda$ is the revenue scaling coefficient. This means with higher $\lambda$, power generation is more profitable and vice versa. Therefore, the agents prioritize locations with higher $E(W(x))$ when $\lambda$ is large. Note that $cE[W(x)]=$ Expected power output in MW at capacity $c$. Since the revenue formulation is in $\$$, the units for $\lambda$ is in $\$/MW$. 

\textbf{Interpretation of $\sigma$:} $\sigma$ is the clustering penalty coefficient. A higher $\sigma$ applies a stronger penalty for clustering and for lower $\sigma$, agents may ignore spatial correlations since there is negligible revenue impact. $Cov(cW(x),cW(x'))$ gives units of $MW^2$, hence $\sigma$ takes units $\$/MW^2$ for dimensional consistency.

While we presently treat $\lambda$ and $\sigma$ as fixed, they may in fact vary with the population distribution $\mu$. For instance, increased renewable penetration could depress market prices, effectively reducing $\lambda$ or increase competition and clustering, making $\sigma$ more relevant. 

For the rest of the discussion, we will use the notation $\tilde{\mu}$ to denote the measure on $\{X^1,...,X^n\}$ defined by   $\int_{\tilde{c}}^{\tilde{C}} c \mu(dc,dx)$ for any probability measure $\mu$ on $A$. Note that
$$F((c_i,X_i),\mu^N)=c_i \tilde{F}(X_i,\tilde{\mu}^N),$$
where $\tilde{\mu}_N(dx)=\int_{\tilde{c}}^{\tilde{C}}c \mu^N(dc,dx)$, and 
\begin{equation}
    \tilde{F}(X,\tilde{\mu}):= \lambda E[W(X)] -\sigma \int Cov(W(X),W(X'))\tilde{\mu}(dX).
\end{equation}

\begin{definition}[MFG Solution]\label{def:MFG_solution}
A probability measure $\mu$ on $A$ is called an MFG solution if:
\begin{enumerate}[label=(\roman*)]
    \item $\sum_{[c^j\in \mathbb{C}]}c^j\mu(c^j,x)\leq C_{max}$ where $\mathbb{C}=\{c^1,...,c^k\}$, $\forall x \in \{X^1,...,X^n\}$. 
    \item For any $x\in X_{max}:=\{x:\sum c^j\mu(c^j,x)=C_{max}\}$ and $c$ such that $\mu(c,x)>0$:
    $$F((c,x),\mu)\geq F((c',y),\mu)\quad \forall\, (c',y)  \mbox{ such that}  \sum c^i \mu(c^i,y)<C_{max}.$$ and
    $$F((c,x),\mu)\geq F((c',x),\mu)\quad \forall\, (c',x)  \mbox{ such that} \, c'<c$$
    \item For any $x\notin X_{max}$ and $c$ such that $\mu(c,x)>0$:
    $$F((c,x),\mu)=\sup_{[c'\in \mathbb{C}, y'\notin X_{max}]} F((c',y'),\mu) $$
\end{enumerate}
\end{definition}

Next, we show that an MFG solution is obtained as the limit of a Nash equilibrium of an $N$-player game, as $N\rightarrow\infty$.  To make this precise, let $\mu^n$ be the empirical distribution of the actions of $N$ agents.  Let $\mu_N ^i [b]$ denote the probability measure obtained by replacing the action of the $i^{th}$ agent with the action $b$.  We say that $\mu_N ^i [b]$ is feasible if it satisfies $\sum_c c\mu_N ^i [b](c,x)\leq C_{max}$, for all $x$.

\begin{definition}[\texorpdfstring{$\epsilon_N$-Nash equilibrium}{Epsilon_N-Nash equilibrium} for the N-player game]\label{def:epsilon_N-Nash}
For any positive sequence of numbers $(\epsilon_N)_{N=0}^{\infty}$, a sequence of action profiles $(a_1^N,...,a_N^N)_{N=0}^\infty$ is called an $\epsilon_N$-Nash equilibrium for the $N$-player game if it satisfies for every $N$:
$$F((a_i^N,\mu^N)\geq F(b,\mu_N ^i [b])-\epsilon_N \quad \forall\, b \mbox{ such that } \mu_N ^i[b] \mbox{ is feasible}.$$  
\end{definition}

The following theorem justifies the definition of an MFG solution.
\begin{theorem}\label{thm:epsilon_N-Nash}
Let $(\epsilon_N)_{N=1}^{\infty}$ be a sequence of positive numbers such that $\lim_{N\rightarrow \infty} \epsilon_N=0$.  Let $(a_1^N,...,a_N^N)$ be $\epsilon_N$-Nash. Then any limit point $\mu$ of  $\mu^N$ as $N \rightarrow\infty$ is an MFG solution.
\end{theorem}
The proof of Theorem \ref{thm:epsilon_N-Nash} is given in Appendix~\ref{app_proofs}.

\begin{theorem}\label{thm:MFE}
Let $\mu$ be an MFG solution.  Let $m$ be a probability measure such that $\tilde{m}$ satisfies (1) and $\sum_x \tilde{m}(x)= \sum_x \tilde{\mu}(x)$. Then,
\begin{equation}
    \sum_{c,x} F((c,x,\mu)m(c,x) \leq \sum_{c,x} F(c,x,\mu)\mu(c,x) \label{mfgineq}
\end{equation}
\end{theorem}

Proof of Theorem ~\ref{thm:MFE} is given in the Appendix~\ref{app_proofs}.

\begin{rmk} \label{rmk:additionalconstraints}

 Theorem ~\ref{thm:MFE} gives a necessary condition for an MFG solution $\mu$.  However, this condition is not sufficient for a measure $\mu$ to be an MFG solution unless additional conditions are satisfied.  For example, when $c^1 < c^k$, to have an MFG solution with the property $\sum_x \tilde{\mu}(x)=\tilde{C}$, in addition to inequality \eqref{mfgineq}, we need $\tilde{F}(x, \tilde{\mu})\geq 0$ for all $x$ such that $\tilde{\mu}>0$, because we need,
 \begin{equation} \label{eq:ck_condition}
    c^k\tilde{F}(X,\tilde{\mu}) \;\geq\; c^1 \tilde{F}(X,\tilde{\mu}).
\end{equation}

 Similarly, in order to have an MFG solution such that $\sum_x \tilde{\mu}(x)=\tilde{c}$, in addition to inequality \eqref{mfgineq}, we need $\tilde{F}(x, \tilde{\mu})\leq 0$ for all $x$ such that $\tilde{\mu}>0$.

 To identify sufficient conditions for the existence of an MFG solution with $\tilde{c}< \sum_x \tilde{\mu}(x)<\tilde{C}$ is more complex and will be explored in a follow up paper. 

 Finally, we note that in the special case $c^1=c^k$, inequality \eqref{mfgineq} is a necessary and sufficient condition for $\mu$ to be an MFG , because \eqref{eq:ck_condition} is always satisfied.

 \end{rmk}

\subsection{\label{sec:Potential_Games}Potential Games}
Since the action space is finite, any probability measure $m$ can be identified with its probability mass function on the action space.  In Table \ref{tab:MFE_masses}, we show how to represent $m$ as an $k\times n$-dimensional vector of non-negative masses.      

\begin{table}[t]
\centering
\begin{tabular}{|c|c|c|c|c|c|c|}
\hline
       & $X_1$ & $X_2$ & $X_3$ & $\cdots$ & $\cdots$ & $X_n$ \\
\hline
$c_1$  & $m_1$ & $m_2$ & $m_3$ & $\cdots$ & $\cdots$ & $m_n$ \\
\hline
$c_2$  & $m_{n+1}$ & $m_{n+2}$ & $m_{n+3}$ & $\cdots$ & $\cdots$ & $m_{2n}$ \\
\hline
$c_3$  & $m_{2n+1}$ & $m_{2n+2}$ & $m_{2n+3}$ & $\cdots$ & $\cdots$ & $m_{3n}$ \\
\hline
$\vdots$ & $\vdots$ & $\vdots$ & $\vdots$ & & & $\vdots$ \\
\hline
$c_k$  & $m_{(k-1)n+1}$ & $m_{(k-1)n+2}$ & $m_{(k-1)n+3}$ & $\cdots$ & $\cdots$ & $m_{kn}$ \\
\hline
\end{tabular}
\caption{Mass allocation matrix across capacities and locations}
\label{tab:MFE_masses}
\end{table}

Let $\mathcal{P}^{C_{max}}$ be the space of all probability measures $m$ on $A=\{c^1,...,c^k\}\times\{X^1,...,X^n\}$ such that $\sum_c cm(c,x)\leq C_{max}$, for all $x$.  We can identify $\mathcal{P}^{C_{max}}$ with the convex subset of $\mathbb{R}^{kn}$ consisting of $(m_1,\ldots, m_{kn})$ such that

\begin{align}
    & \sum_i c_i m_{(i-1)n+j}\leq C_{{max}} \quad \forall j=1,...,n, \\
    & \sum_{i=1}^k\sum_{j=1}^nm_{(i-1)n+j}=1,\\
    & m_l \geq 0 \,\, \forall \,l=1,\ldots, kn. 
\end{align}

A potential game is a game where there is a single scalar-valued function $G$, called the potential function, defined on $\mathcal{P}^{C_{max}}$, such that  $\nabla G (m)=   (F_1(m), \dots, F_{nk}(m))$, where $F_l(m)=F(a_l,m)$ and $a_l=(c^i,X^j)$ for $l=(i-1)n+j$.

In our case, we can show that there is a potential function $G$ and it is given by
\begin{equation}
\begin{aligned}
  G(m) &= \lambda \sum_{i=1}^k c^i \left[ \sum_{j=1}^n \mathbb{E}[W(X^j)] \, m_{(i-1)n + j} \right] \\
&-\frac{\sigma}{2}\sum_{i=1}^k (c^i)^2 \left[ \sum_{j=1}^n Cov(X^j, X^j) \, m_{[(i-1)n + j]} \, m_{[(i-1)n + j]}\right]\\
&\quad - \frac{\sigma}{2} \sum_{i=1}^k (c^i)^2 \left[ \sum_{j=1}^n \sum_{l \neq j }^n Cov(X^j, X^l) \, m_{[(i-1)n + j]} \, m_{[(i-1)n + l]} \right] \\
&\quad - \frac{\sigma}{2} \sum_{i=1}^k \sum_{p \neq i}^k c^i c^p \left[ \sum_{j=1}^n \sum_{l=1}^n Cov(X^j, X^l) \, m_{[(i-1)n + j]} \, m_{[(p-1)n + l]} \right].
\end{aligned}
\end{equation}
Indeed, 
\begin{eqnarray*}
    (\nabla G(m))_l&=& \lambda c_i \mathbb{E}[W(X^j)]-\sigma (c^i)^2 Cov(X^j,X^j)m_{l}\\
    &&-\sigma\sum_{j'\neq j} (c^{i})^2 Cov(X^j, X^{j'})m_{(i-1)n+j'}\\
    &&-\sigma \sum_{p\neq i}^k c^i c^p \sum_{j'=1}^{n} Cov(X^j,X^{j'})m_{(p-1)n+j'} \\
    &=&c_i\left[\lambda \mathbb{E}[W(X^{j})]-\sigma \sum_{j'=1}^{n} Cov(X^j,X^{j'})\sum_{i=1}^k c^i m_{(i-1)n+j'}\right]\\
    &=&F((c^i,X^j),m).
    \end{eqnarray*}

To simplify $G(m)$, we define a new variable $w_i$ by : 
\begin{equation}\label{eq:AvgCapacity}
w_i =c_1m_i+c_2m_{n+i}+...+c_km_{(k-1)n+i}
\end{equation}
 The variable $w_i$ can be interpreted as the \textit{total capacity} invested at location $X^i$. 
We observe that
\begin{equation}
    G(m)=\tilde{G}(w) = -\frac{1}{2}  \sigma \, \boldsymbol{w}^\top \mathbf{C} \, \boldsymbol{w} + \lambda \, \boldsymbol{E}[W]^\top \boldsymbol{w}
    \label{eq:potentialproblemweight}
\end{equation}
where $\mathbf{C}$ is the matrix with $\mathbf{C}_{i,j}=Cov(X^i,X^j)$ and $\boldsymbol{E}[W]$ is the column vector with $\boldsymbol{E}[W]_i= \mathbb{E}[W(X^i)]$, and $\boldsymbol{w}$ is the column vector with $\boldsymbol{w}_i=w_i$. 

\begin{theorem} \label{thm:potential} Let $\mu$ be an MFG solution such that $\sum_x\tilde{\mu}(x)=C\in [\tilde{c},\tilde{C}]$ and $\tilde{F}(x,\tilde{\mu})\geq 0$, for $x$ such that $\tilde{\mu}(x)=C_{max}$.  Let $\boldsymbol{w}^*$ be the column vector with $\boldsymbol{w}_i^*= \tilde{\mu}(X^i)$. Then $\boldsymbol{w}^*$ is the unique solution of
\begin{eqnarray}
    &&\max -\frac{1}{2}  \sigma \, \boldsymbol{w}^\top \mathbf{C} \, \boldsymbol{w} + \lambda \, \boldsymbol{E}[W]^\top \boldsymbol{w}\label{optimize}
\end{eqnarray}
\textit{Subject to:}
\[
\begin{aligned}
    &\sum_{i=1}^n w_i = C \\
    &0 \leq w_i \leq C_{max}  \quad \mbox{for} \quad i = 1, \ldots, n
\end{aligned}
\]
\end{theorem}

To prove this theorem, we first need the following lemma.
\begin{lemma}\label{thm:Farkas_Lemma}

 For any  $w=(w_1, \dots, w_n)$  such that
\begin{align}
    & \tilde{c}\leq \sum_{j=1}^n w_j \leq \tilde{C}   , \quad 0 \leq w_j \leq \tilde{C} \text{ for } j= 1,\dots,n, \nonumber
\end{align} there exists a vector $m = (m_1, \dots, m_{kn})$ such that 
\begin{align}
& \qquad \qquad \sum_{i=1}^k\sum_{j=1}^nm_{(i-1)n+j}=1,\nonumber\\
&\qquad\qquad \tilde{m}_j= \sum_{i=1}^k c^i m_{(i-1)n+j}\nonumber\\
&\qquad\qquad 0\leq m_l \,\mbox{ for } l= 1,\dots,kn. \nonumber
\end{align}
\end{lemma}

The proof of Lemma \ref{thm:Farkas_Lemma} is based on Farkas Lemma and deferred to Appendix A. 
\begin{proof} \textbf{(of Theorem \ref{thm:potential})}

We follow a similar argument as in \cite{Lacker2018Notes}: Theorem 2 implies that $\frac{d}{d\epsilon}\langle \nabla \tilde{G}(\tilde{\mu}), \tilde{\mu}+\epsilon(\tilde{m}-\tilde{\mu})\rangle = \sum_x\tilde{F}(x,\tilde{\mu})\tilde{m}(x)-\sum_x\tilde{F}(x,\tilde{\mu})\tilde{\mu}(x)\leq 0$ for all $m$ such that $0\leq \tilde{m}(x)\leq C_{max}$ for all $x$, and $\sum_{x}\tilde{m}(x)= C$.  

Let ${P_1}$ be the set of $\tilde{m}=(\tilde{m}_1,\ldots \tilde{m}_n)$ such that 
\begin{align}
     &\qquad\qquad \tilde{m}_j= \sum_{i=1}^k c^i m_{(i-1)n+j}\\
     &\qquad\qquad \sum_{x}\tilde{m}(x)= C\\ 
     & \qquad \qquad \sum_{i=1}^k\sum_{j=1}^nm_{(i-1)n+j}=1, \quad m_i \geq 0 \nonumber\\
     &\qquad\qquad 0\leq \tilde{m}(x)\leq C_{max}
\end{align}

Let $P_2$ be the set of $w=(w_1,\ldots w_n)$ such that
\begin{align}
    & \qquad \sum_{j=1}^n w_j =C   , \quad 0 \leq w_j \leq C_{max} \text{ for } j= 1,\dots,n \nonumber\\
\end{align}

Due to Lemma \ref{thm:Farkas_Lemma}, we have that $P_1=P_2$.  Hence we conclude that, if $\mu$ is an MFG solution then $\tilde{\mu}\in P_2$ and $\frac{d}{d\epsilon}\langle \nabla \tilde{G}(\tilde{\mu}), \tilde{\mu}+\epsilon(w-\tilde{\mu})\rangle$ for all $w$ in $P_2$

Hence $\boldsymbol{w}^*$ is a local maximum of the objective function \eqref{optimize} on $P_2$. Since the objective function \eqref{optimize} is concave and $P_2$ is convex, the local maximum is the global maximum.
\end{proof}




Our strategy to find an MFG solution is to solve the optimization problem \eqref{optimize} and further validate that solution satisfies the additional conditions required for the MFG solution (see Remark \ref{rmk:additionalconstraints} ).  Since the additional conditions are straightforward for the case $C=\tilde{C}$, in this paper, we focus on MFG solutions $\mu$ with $\sum_x\tilde{\mu}(x)=\tilde{C}$.  Note that in this solution, all the investors invest at the maximum allowed capacity.  Another significance of this case is that we can clearly establish that for the case $c^1=c^K$, the solution of \eqref{optimize} on $P_2$ is the unique MFG  solution.  

To solve \eqref{optimize} on $P_2$, we follow the \textit{Semidefinite Programming} techniques introduced by \citet{semidef_prog} which specifically falls under the section \textit{Quadratic Programming} in this literature and we use the R library \textit{quadprog} modeled by \citet{quadprog} to solve our maximization problem.

A typical quadratic programming problem takes the following form:

\begin{equation}
    \min\limits_{x} \left( \frac{1}{2} X^T D X - d^T X \right)
\end{equation}
\textit{subject to constraints written in matrix form : $A^TX \geq b$}
\newline
If we convert this minimization problem to a maximization we immediately have the form of \eqref{optimize} on $P_2$ in the matrix form as shown above.


\section{\label{sec:RevenueModelWindFarms}Revenue model for Windfarms in Alberta}
In this chapter, we explore the relationship between wind resources and revenue generation empirically using power from \citet{AESO_MeteredVolumes} and price data obtained from \citet{AESO_PoolPrice}. This analysis uses market rules and historical data spanning the years 2016 to 2024. 

We expect that a high capacity factor boosts revenue while spatial crowding of wind farms reduces revenue due to the merit order effect (see section \ref{sec:crowding_validation}).
To investigate this, we begin by examining how wind farms in Alberta participate in the electricity market and earn revenue under the current market structure. In Alberta, electricity generators are paid the hourly spot price multiplied by the power they have generated. A tax of 37 cents per megawatt hour of electricity generated is paid by generators to AESO to operate the market \citep{AESO_TradingCharge2024}.
Given this context, we aim to estimate the empirical revenue earned by each existing wind farm in Alberta. This involves calculating revenue based on hourly generation data and spot prices. These revenue estimates will then be used in a linear regression framework to identify and quantify the relationship between revenue and wind resource quality across different sites. 

To calculate the annual revenue of wind farms in Alberta, consider the following parameters:
\begin{itemize}
\item $t = 1,2,...T$ represents the hours in a year. For leap years, $T = 8784$. For a non-leap year, $T = 8760$.
\item $\tau = 2016,...., 2024$ represents the year.
\item $i$ corresponds to agents producing power from wind energy.
\item $s_{t}^{\tau}$ is the pool price at hour $t$ year $\tau$.
\item $P_{t} ^ {\tau,i}$ is the power produced by agent $i$ at any hour $t$ of a given year $\tau$.
\item $c_i$ is the total capacity installed by agent $i$.
\end{itemize}

The empirical hourly revenue of wind farn $i$, $r^i _t$  is calculated using the data by:
\begin{equation}
    r^{\tau,i} _t = P_{t} ^ {\tau,i}*s_{t}^{\tau} - 0.37*P_{t} ^ {\tau,i}
\end{equation}

\noindent The yearly normalized empirical revenue per capacity for wind farm $i$, $R^{i,\tau}$ is then:
\begin{equation}\label{eq:normalized_revenue}
    R^{i,\tau} = \frac{\sum_{t=1}^{T} r^{i, \tau}_{t}}{\bar{S^\tau} c_i}
\end{equation}
where $\bar{S^\tau}=\frac{\sum_{t=1}^{T}s_{t}^{\tau}}{T}$ is the average pool price for year $\tau$.

Table \ref{tab:normrevcap} shows the average yearly normalized revenue per capacity of wind farms in Alberta from 2016-2024. 
We represent power output as a function of capacity factor, $W(x)$ , where $W(x)$ is defined as equation~\eqref{eq:CapacityFactorEq}.

\begin{table}[htbp]
\centering
\caption{Normalized Revenue Per Capacity by Farm (NRPC): (2016--2024)}
\label{tab:normrevcap}
\resizebox{\textwidth}{!}{%
\begin{tabular}{clrrrrrrrrr}
\toprule
 FarmName & NRPC\_2016 & NRPC\_2017 & NRPC\_2018 & NRPC\_2019 & NRPC\_2020 & NRPC\_2021 & NRPC\_2022 & NRPC\_2023 & NRPC\_2024 \\
\midrule
 AKE1  & 2430.027 & 2389.474 & 2064.204 & 1730.00 & 2206.551 & 1635.539 & 1496.08 & 1375.657 & 1370.13 \\
 ARD1  & 2730.693 & 2677.119 & 2265.105 & 1968.38 & 2478.617 & 1858.822 & 1769.70 & 1607.062 & 1672.46 \\
 BSR1  & 2981.979 & 2888.961 & 2180.911 & 2055.13 & 2615.075 & 2490.024 & 2318.35 & 1773.314 & 1661.74 \\
 BTR1  & 2483.548 & 2415.789 & 2122.224 & 1825.88 & 2313.728 & 1816.250 & 1647.72 & 1487.677 & 1556.20 \\
 BUL12 & 3265.324 & 3573.209 & 3128.339 & 3231.51 & 3424.404 & 3078.781 & 3078.07 & 2620.131 & 2116.63 \\
 CR1   & 2321.991 & 2294.687 & 2019.323 & 1681.66 & 2076.240 & 1662.321 & 1465.76 & 1445.197 & 1423.10 \\
 CRR1  & 2632.423 & 2548.415 & 2269.823 & 1813.37 & 2316.492 & 1944.560 & 1659.52 & 1502.609 & 1648.94 \\
 GWW1  & 2741.881 & 2669.258 & 2272.821 & 1830.64 & 2464.500 & 1946.054 & 1925.79 & 1722.013 & 1843.18 \\
 HAL1  & 2748.259 & 3037.183 & 2428.509 & 2385.69 & 2769.479 & 2712.487 & 2544.41 & 1942.182 & 1585.81 \\
 IEW1  & 2526.847 & 2564.230 & 2148.510 & 1701.64 & 2137.914 & 1726.325 & 1513.08 & 1471.161 & 1416.76 \\
 IEW2  & 2224.736 & 2205.126 & 1882.285 & 1591.29 & 2010.876 & 1509.429 & 1394.20 & 1356.647 & 1227.53 \\
 KHW1  & 2659.561 & 2648.342 & 2342.797 & 1955.09 & 2414.713 & 1881.680 & 1748.59 & 1632.810 & 1699.42 \\
 NEP1  & 2043.848 & 2269.060 & 1670.642 & 1592.01 & 1930.510 & 1898.678 & 1982.54 & 1301.872 & 1198.94 \\
 OWF1  & 2960.188 & 2796.227 & 2433.605 & 2023.14 & 2510.779 & 1951.694 & 1813.99 & 1766.302 & 1680.48 \\
 SCR2  & 2629.910 & 2497.751 & 2041.735 & 1873.38 & 2471.758 & 1829.938 & 1864.55 & 1587.423 & 1488.59 \\
 SCR3  & 2759.867 & 2616.168 & 1950.466 & 1849.65 & 2213.759 & 1980.939 & 1860.95 & 1499.202 & 1349.54 \\
 SCR4  & 2895.011 & 3065.141 & 2329.019 & 2144.77 & 2555.770 & 2622.038 & 2570.57 & 1578.494 & 1489.15 \\
 TAB1  & 2413.541 & 2356.347 & 1869.184 & 1732.65 & 2099.143 & 1880.846 & 1607.34 & 1216.333 & 1002.97 \\
\midrule
\textbf{Avg Pool Price (\$/MWh)} & \textbf{18.28} & \textbf{22.19} & \textbf{50.35} & \textbf{54.88} & \textbf{46.72} & \textbf{101.93} & \textbf{162.46} & \textbf{133.63} & \textbf{62.78} \\
\bottomrule
\end{tabular}%
}
\end{table}

Revenue is dependent on power produced. Therefore, to estimate potential revenue from wind resources at any location within Alberta, we must develop a model that approximates wind power from wind speed (see section \ref{sec:CapFactorEst}). 

\section{Empirical Validation of Covariance as a Penalty}
\label{sec:crowding_validation}

A key prediction of our mean-field reward function is that the revenue of a wind farm is penalized by its covariance with the rest of the fleet. Wind farms whose revenues are highly correlated with total provincial generation should, on average, receive systematically lower normalised revenues as a result of the merit order effect when many farms produce simultaneously at high rates, driving down the price. 

Before applying the mean-field game approach for optimizing wind farm locations, we also check whether the "crowding penalty" is indeed empirically observed in the historical Alberta data. To do so, we directly estimate $\lambda$ and $\sigma$ from the observed revenues, realized capacity factors, and empirical
covariances.

\subsection{Construction of Regression Variables}\label{sec:Regression_construction}

All three variables are constructed from empirical data for each farm-year observation.

\paragraph{Normalised revenue} : $\tilde{R}_{i,y}$ is the dependent variable where it is calculated by dividing the annual revenue of farm $i$ in year $y$ by both its installed capacity $c_i$ and the annual average pool price where the annual revenue values from section \ref{sec:RevenueModelWindFarms} are used here.

Here $\bar{P}_y$ denotes the average pool price for year $y$.

\begin{equation}
    \tilde{R}_{i,y}
    \;=\;
    \frac{\mathrm{Rev}_{i,y}}{c_i \,\bar{P}_y}.
    \label{eq:norm_rev}
\end{equation}

\noindent
Dividing by the average pool price removes annual price fluctuations,
so that $\tilde{R}_{i,y}$ reflects the performance of the
farm rather than market conditions. 

\paragraph{Mean capacity factor.} $\mathbb{E}[W(x)_{i,y}]$ is the empirical annual mean capacity factor of farm $i$, computed directly from the hourly AESO dispatch data.
\begin{equation}
    \mathbb{E}[W(x)_{i,y}]
    \;=\;
    \frac{1}{T_{i,y}}
    \sum_{t=1}^{T_{i,y}}
    W(x)_{i,t},
    \label{eq:empirical_cf}
\end{equation}
\noindent
where $W(x)_{i,t} = P_{i,t}/C_i$ is the hourly capacity factor,
$P_{i,t}$ is the observed power output, $C_i$ is the installed capacity,
and $T_{i,y}$ is the number of operating hours of farm $i$ in year $y$.

\paragraph{Covariance per capacity.}
Let $\mathbf{W}_y$ denote the $T_y \times N_y$ matrix of hourly capacity factors for all $N_y$ operating farms in year $y$, where each column corresponds to a farm and each row to an hourly observation. Since not all farms operate for the full $T_y$ hours in a given year, each entry $(i,j)$ of the empirical covariance matrix is computed using only the $n_{ij}$ hours in which both farms $i$ and $j$ report non-missing output:
\begin{equation}
    \widehat{\Sigma}_{y,ij}
    \;=\;
    \frac{1}{n_{ij} - 1}
    \sum_{t \in \mathcal{T}_{ij}}
    \Bigl(W(x)_{i,t} - \bar{W}(x)_{i}^{(ij)}\Bigr)
    \Bigl(W(x)_{j,t} - \bar{W}(x)_{j}^{(ij)}\Bigr),
    \label{eq:emp_cov}
\end{equation}
\noindent
where $\mathcal{T}_{ij} = \{t : W(x)_{i,t} \text{ and } W(x)_{j,t} \text{ are both observed}\}$, $n_{ij} = |\mathcal{T}_{ij}|$, and $\bar{W}(x)_{i}^{(ij)}$ denotes the mean capacity factor of farm $i$ computed over $\mathcal{T}_{ij}$ only.

Our final regression model is now,

\begin{equation}
\tilde{R}_{i,y} = \hat{\lambda}\,\mathbb{E}[W(x)_{i,y}] -
\hat{\sigma}\,\sum_{j=1}^{N_y} c_j\,\widehat{\Sigma}_{y,ij}
\label{eq:regression}
\end{equation}

\subsection{Results and Interpretation} \label{Parameter_estimation}

Using the empirical values calculated in Section~\ref{sec:Regression_construction} from AESO power data and applying it to equation \ref{eq:regression} we numerically estimate our parameters using linear regression. Table~\ref{tab:regression_results} reports the estimates across all windows.
\begin{table}[htbp]
\centering
\caption{Rolling five-year OLS estimates of $\hat{\lambda}$ and $\hat{\sigma}$
         from~\eqref{eq:regression}, using empirical capacity factors
         and covariances from raw hourly AESO dispatch data.
         Standard errors in parentheses. All estimates significant
         at the 1\% level.}
\label{tab:regression_results}
\small
\begin{tabular}{lrrrrrr}
\hline
\textbf{Window} & $\hat{\lambda}$ & $(\mathrm{SE})$ &
$\hat{\sigma}$ & $(\mathrm{SE})$ & $R^2$ & $n$ \\
\hline
2016--2020 & 8{,}285 & (287.0) & $5.186$ & (0.945) & 0.707 &  96 \\
2017--2021 & 8{,}218 & (265.7) & $5.878$ & (0.797) & 0.743 & 100 \\
2018--2022 & 7{,}986 & (203.2) & $5.940$ & (0.553) & 0.769 & 108 \\
2019--2023 & 7{,}540 & (172.5) & $4.879$ & (0.384) & 0.819 & 127 \\
2020--2024 & 7{,}326 & (162.6) & $4.375$ & (0.317) & 0.818 & 152 \\
\hline
All years  & 8{,}281 & (133.2) & $5.901$ & (0.292) & 0.836 & 227 \\
\hline
\end{tabular}
\end{table}

The key empirical finding is that $\hat{\sigma}$ is positive and highly significant across every estimation window, ranging from $4.37$ to $5.94$. This confirms the central prediction of our mean-field game framework: farms whose hourly output is highly correlated with the rest of the Alberta fleet earn systematically lower revenue per unit capacity, even after controlling for their own mean generation through $\bar{\mathrm{CF}}_{i,y}$. This mechanism is the merit-order effect, where simultaneous high output from spatially clustered farms suppresses the pool price.

The parameter $\hat{\lambda}$ ranges from $7326$ to $8285$ across windows, reflecting variation in the effective price of generation over the sample period after pool-price normalisation. Its stability across windows indicates that the relationship between mean capacity factor and normalised revenue is consistent over time. The $R^2$ values range from $0.71$ to $0.84$, with the full-sample fit indicating that mean capacity factor and spatial covariance exposure together explain $\approx$ 84\% of the cross-sectional and inter-temporal variation in normalized farm revenue.

These empirically estimated parameters $\hat{\lambda} = 8281$ and $\hat{\sigma} = 5.901$ (full sample) confirm that the reward function in the mean-field game is correctly specified and that spatial co-production is a genuine economic cost. We adopt the full-sample estimates as inputs to the static mean-field game solved in Section~\ref{sec:mfg_solution}.

\section{\label{sec:CapFactorEst}Expected Capacity Factor Estimation}
For the estimation of the expected capacity factor, climate data were obtained from the Alberta Climate Information 
Service (ACIS) \citep{ACIS2024}.  To estimate the power output from wind speed data, two main challenges must be addressed:

\begin{enumerate}
\item Wind speeds are recorded at a reference height that typically differs from the actual hub height of wind turbines.
\item Wind speed measurements are taken from stations that are not located precisely at wind farm sites.
\end{enumerate}

In the following sections, we describe the methods used to address each of these challenges.
\subsection{Log-wind Profile}\label{Log_wind_profile}

Wind speed measurements by ACIS are obtained at a height of 10 meters \citep{ACIS2024}. Since wind turbines are at varying heights, much greater than 10 meters, the wind speed must be scaled to the location of the hub heights.

The equation to estimate mean wind speed ($u_{z}$) at a height z (meters) above the ground is given by the log-wind profile \citet{Oke_1987}:
\begin{equation}
    u_z = \frac{u_*}{k} \left[ln \left(\frac{z-d}{z_0} \right) + \phi(z,z_0,L ) \right]
\end{equation}
where $u_{*}$ is the friction velocity ($ms^{-1}$), $\kappa$ is the Von Karman constant(0.41), $d$ is the zero plane displacement (meters), $z_0$ is the surface roughness (meters), $\Phi$ is the stability term where $L$ is the Obukhov length. 
Under neutral stability conditions, $\frac{z}{L}$ = 0 and $\Phi$ drops out.  The equation simplifies to :
\begin{equation}
    u_z = \frac{u_*}{k} \left[ln(\frac{z-d}{z_0})\right]
\end{equation}

The equation can be rearranged to estimate the mean wind speed at one height $(z_2)$ based on another $(z_1)$ :
\begin{equation}
    u(z_{2}) =  u(z_{1}) \frac{ln((z_2 - d)/z_0)}{ln((z_1 - d)/z_0)}
\end{equation}
where u($z_1$) is the mean wind speed at height $z_1$.  Zero-plane displacement (d) is the height in meters above the ground at which zero mean wind speed is achieved as a result of flow obstacles such as trees or buildings.

Since wind farms are located in agricultural/rural areas, we will assume that there are no obstacles (d=0). The equation reduces to:
\begin{equation}\label{eq:WSA_height_formula}
    u(z_{2}) =  u(z_{1}) \frac{ln(z_2/z_0)}{ln(z_1/z_0)}.
\end{equation}

The values of $z_0$ were estimated at each wind farm location using ERA5 reanalysis data \citep{ERA5}, which provides wind speeds at both 10\,m and 100\,m heights. Inverting equation~\eqref{eq:WSA_height_formula} with $z_1 = 10$\,m and $z_2 = 100$\,m, we solve for $z_0$ at each location as:
\begin{equation}\label{eq:z0_estimation}
    z_0 = \exp\left(\frac{u(z_1)\ln(z_2) - 
    u(z_2)\ln(z_1)}{u(z_1) - u(z_2)}\right).
\end{equation}

\subsection{Justification of the Sigmoid Model from Wind Speed to Wind Power}
\label{sec:power_curve}

Now, we address our second challenge where we estimate the capacity factor at an unknown location (grid point) in Alberta. To do this, we use the power volumes provided by \citet{AESO_MeteredVolumes} for each wind producer in Alberta. For each wind farm, we consider weather stations within a radius of 50km. Table \ref{tab:training_results} provides the distance to the closest weather station and the number of stations within a radius of 50 km for each wind farm considered during the period. These metrics play an important role since they affect the prediction quality of our wind speed to wind power mapping, as illustrated in section~\ref{sec:sigmoid_estimation_results}. 

Using equation~\eqref{eq:WSA_height_formula} with $z_1 = 10m$ and $z_2$ being the hub height of the turbine at each wind farm, and the $10m$ wind speeds obtained from ACIS weather stations \citep{ACIS2024} with the site-specific $z_0$ estimates obtained from equation~\eqref{eq:z0_estimation}, we obtain hourly wind speeds at each weather station adjusted to the hub height of each wind farm. Next, applying an inverse distance weight to these hourly wind speeds, we find approximate wind speed averages for our wind farm. Figure~\ref{fig:Windspeed_vs_Power Plots} illustrates the power curves for four representative farms spanning a range of capacities and station coverage levels.

\subsubsection{Data Exclusions}
\label{sec:data_exclusions}
During the years 2016-2019, Alberta had 19 operational wind farms. However, we excluded two farms from the analysis on data quality grounds. Farm CRE3 underwent a capacity derating from 38\,MW to 20\,MW in March 2016, experienced frequent extended maintenance outages (multiple shutdowns per year during 2016-2018), and had no power data for August-December 2019. These irregularities compromise the reliability of the empirical power curve at this site. Farm HAL1, located in the northeast of the province where the nearest weather station is more than 46 km away with only one station within 50 km, was excluded due to insufficient wind speed coverage for reliable estimation. The remaining 17 farms constitute as our validation sample.

\begin{figure*}[t] 
\centering
    \begin{subfigure}[b]{0.48\textwidth}
        \centering
        \includegraphics[width=\textwidth]{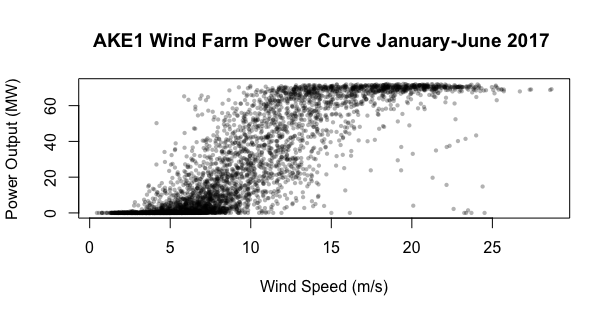}
        \caption{AKE1 (73\,MW) - high $R^2 = 0.789$, dense nearby stations.}
        \label{fig:pc_ake1}
    \end{subfigure}
    \hfill
    \begin{subfigure}[b]{0.48\textwidth}
        \centering
        \includegraphics[width=\textwidth]{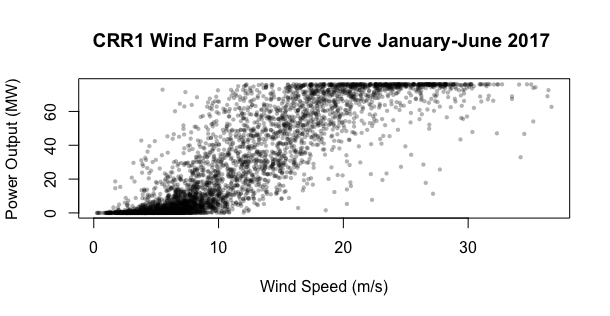}
        \caption{CRR1 (77\,MW) -=- high $R^2 = 0.776$, dense nearby stations.}
        \label{fig:pc_crr1}
    \end{subfigure}

    \vspace{0.5em}

    \begin{subfigure}[b]{0.48\textwidth}
        \centering
        \includegraphics[width=\textwidth]{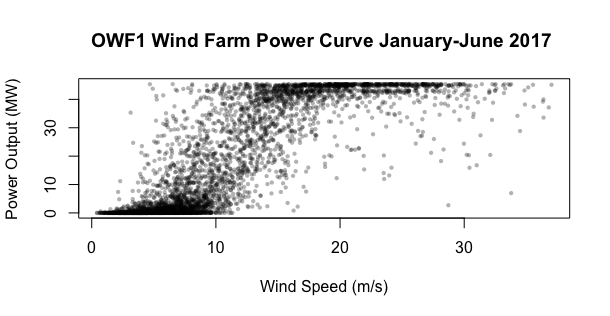}
        \caption{OWF1 (46\,MW) - moderate $R^2 = 0.740$, intermediate station coverage.}
        \label{fig:pc_owf1}
    \end{subfigure}
    \hfill
    \begin{subfigure}[b]{0.48\textwidth}
        \centering
        \includegraphics[width=\textwidth]{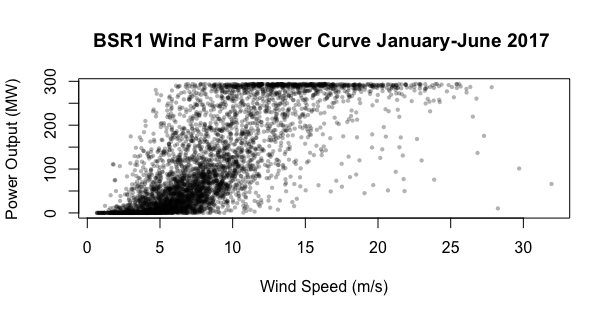}
        \caption{BSR1 (300\,MW) - lower $R^2 = 0.604$, far away stations.}
        \label{fig:pc_bsr1}
    \end{subfigure}

\caption{Wind speed versus power output (January-June 2017) for four
         representative Alberta wind farms spanning a range of capacities and
         goodness-of-fit values. The S-shaped ramp-up from cut-in to rated power
         is visible in all farms, motivating the logistic sigmoid model in
         \eqref{eq:sigmoid}. Variation in scatter across panels reflects
         differences in wind speed estimation quality, which depends on the
         proximity and density of surrounding weather stations rather than a
         deficiency of the sigmoid parameterisation itself.}
\label{fig:Windspeed_vs_Power Plots}
\end{figure*}

Wind turbine power output as a function of wind speed is typically known to follow a characteristic S-shaped curve. We have zero generation below the cut-in wind speed, a steep rise through the operational range, and saturation at the rated or maximum capacity above the rated wind speed. This S-shaped relationship between wind speed and power output is clearly visible across all panels of Figure~\ref{fig:Windspeed_vs_Power Plots}. The contrast between AKE1 and CRR1 (tight scatter, $R^2 \geq 0.77$) versus BSR1 (wider scatter, $R^2 = 0.60$) reflects differences in wind speed interpolation quality rather than model misspecification.

To capture this behaviour, we model the capacity factor $\mathrm{CF}(v)$ of each wind farm using a constrained logistic sigmoid function of the form
\begin{equation}
    \mathrm{CF}(v) \;=\;
    \begin{cases}
        0, & v < v_{\min}, \\[4pt]
        \displaystyle\frac{1}{1 + \exp\!\bigl(-\beta\,(v - \alpha)\bigr)},
        & v_{\min} \leq v \leq v_{\max}, \\[6pt]
        1, & v > v_{\max},
    \end{cases}
    \label{eq:sigmoid}
\end{equation}
where $v$ denotes the hub-height wind speed (m\,s$^{-1}$), $\alpha$ is the inflection-point wind speed at which the farm reaches 50\% of its observed maximum output, and $\beta > 0$ governs the steepness of the ramp-up. The cut-in threshold $v_{\min}$ and rated-power threshold $v_{\max}$ are determined empirically from the power data for each farm--period combination (see Section~\ref{sec:threshold_estimation}), and the asymptotes are fixed at 0 and 1 so that the two free parameters $(\alpha, \beta)$ describe only the shape of the ramp-up region.

The normalised capacity factor is defined as
\begin{equation}
    \mathrm{CF}_{i,t} \;=\; \frac{P_{i,t}}{P_{i}^{\max}},
    \label{eq:cf_def}
\end{equation}
where $P_{i,t}$ is the hourly generation of farm $i$ at time $t$ and $P_{i}^{\max}$ is the empirical maximum observed output for that farm-period. Normalizing by $P_{i}^{\max}$ rather than nameplate capacity decouples the curve-shape parameters from the rated capacity of the farm, making $\alpha$ and $\beta$ directly comparable across farms of different sizes.

\subsubsection{Threshold Estimation}
\label{sec:threshold_estimation}

The cut-in wind speed $v_{\min}$ and rated wind speed $v_{\max}$ are estimated directly from the observed generation data for each farm--period, without requiring manufacturer specifications. Let $n$ denote the number of observed hourly records in the fitting window and define
\begin{align}
    n_{\mathrm{low}} &= \#\{t : P_{i,t} < 1\,\text{MW}\}, \notag \\
    n_{\mathrm{high}} &= \#\{t : P_{i,t} > P_{i}^{\max} - 2\,\text{MW}\}. \notag
\end{align}
The corresponding quantile probabilities are $p_0 = \frac{n_{\mathrm{low}}}{n}$ and $p_1 = 1 - \frac{n_{\mathrm{high}}}{n}$, and the thresholds are set to
\begin{equation}
    v_{\min} = Q_{p_0}(v), \qquad v_{\max} = Q_{p_1}(v),
    \label{eq:thresholds}
\end{equation}
where $Q_p(v)$ denotes the $p$-th empirical quantile of the wind speed observations. This approach ensures that the sigmoid is fitted only on hours where the farm is in its operational ramp-up regime, excluding both low-wind shutdown periods and high-wind saturation periods.

\subsubsection{Parameter Estimation}
\label{sec:param_estimation}

For each farm $i$, year $y$, and six-month period $d \in \{\text{January-June},\,\text{July-December}\}$, the parameters $(\alpha_{i,y,d},\, \beta_{i,y,d})$ are estimated by non-linear least squares on the capacity-factor observations in the operational window $[v_{\min},\, v_{\max}]$:
\begin{equation}
(\hat\alpha,\,\hat\beta) \;=\;
\underset{\alpha,\,\beta}{\arg\min}
    \sum_{t:\,v_{\min} \leq v_t \leq v_{\max}}
    \!\!\!\!\bigl(\mathrm{CF}_{i,t} - \sigma(v_t;\,\alpha,\beta)\bigr)^2,
    \label{eq:nls}
\end{equation}
where $\sigma(v;\alpha,\beta) = [1 + \exp(-\beta(v-\alpha))]^{-1}$.

We fit each wind farm separately for each six-month period to capture seasonal variation in the power curve arising from differences in air density and the prevailing wind direction. Extended periods of zero generation exceeding 60 consecutive hours are removed prior to fitting to exclude planned maintenance outages, which are unrelated to the wind speed--power relationship.

\subsubsection{Training Results: In-Sample Fit (2016-2019)}
\label{sec:sigmoid_estimation_results}

Table~\ref{tab:training_results} reports the averaged sigmoid parameters and in-sample goodness-of-fit statistics for the 17 retained wind farms over the training period 2016-2019. Each farm was fitted over eight semi-annual windows (four years $\times$ two periods), and the reported values are the means across these eight fits. The table also includes the distance to the nearest weather station and the number of stations within 50\,km to illustrate the relationship between station coverage and model performance.

\begin{table}[htbp]
\centering
\caption{Averaged sigmoid parameters and in-sample goodness-of-fit over the training period (2016--2019, eight semi-annual windows per farm). $\bar{v}_{\min}$, $\bar{v}_{\max}$: mean cut-in and rated wind speed thresholds (m\,s$^{-1}$); $\bar{\alpha}$: mean inflection-point wind speed (m\,s$^{-1}$); $\bar{\beta}$: mean ramp-up steepness; $\overline{\mathrm{ACF}}$, $\overline{\mathrm{ECF}}$: mean actual and estimated capacity factors; $\overline{R^2}$: mean coefficient of determination; $d_{\mathrm{stn}}$: distance to nearest weather station (km); $n_{\mathrm{stn}}$: number of stations within 50\,km. Farms are sorted by $\overline{R^2}$ (descending).}
\label{tab:training_results}
\small
\begin{tabular}{lrrrrrrrrrrr}
\hline
\textbf{Farm} & \textbf{Capacity} & $\bar{v}_{\min}$ & $\bar{v}_{\max}$ & $\bar{\alpha}$ & $\bar{\beta}$ & $\overline{\mathrm{ACF}}$ & $\overline{\mathrm{ECF}}$ & $\overline{R^2}$ & $\bar{r}$ & $d_{\mathrm{stn}}$ & $n_{\mathrm{stn}}$ \\
\hline
CRR1  &  77 &  5.57 & 22.4 & 13.37 & 0.306 & 0.334 & 0.337 & 0.787 & 0.889 &  4.79 & 6  \\
CR1   &  39 &  6.12 & 20.8 & 12.33 & 0.337 & 0.304 & 0.305 & 0.780 & 0.887 &  3.28 & 5  \\
ARD1  &  68 &  5.76 & 22.1 & 12.53 & 0.354 & 0.351 & 0.357 & 0.769 & 0.879 &  2.77 & 7  \\
AKE1  &  73 &  5.91 & 20.2 & 10.96 & 0.410 & 0.329 & 0.333 & 0.757 & 0.872 & 10.74 & 7  \\
KHW1  &  63 &  5.89 & 20.1 & 10.97 & 0.362 & 0.355 & 0.358 & 0.745 & 0.867 & 11.09 & 6  \\
OWF1  &  46 &  6.11 & 21.5 & 11.53 & 0.372 & 0.362 & 0.365 & 0.733 & 0.862 &  5.48 & 6  \\
IEW2  &  66 &  6.20 & 25.3 & 14.14 & 0.279 & 0.292 & 0.293 & 0.718 & 0.851 &  4.58 & 6  \\
GWW1  &  71 &  5.37 & 22.7 & 11.65 & 0.369 & 0.347 & 0.359 & 0.717 & 0.852 &  7.51 & 7  \\
BTR1  &  66 &  5.73 & 22.1 & 11.89 & 0.332 & 0.327 & 0.331 & 0.713 & 0.846 & 14.94 & 7  \\
IEW1  &  66 &  6.17 & 26.1 & 12.27 & 0.283 & 0.331 & 0.329 & 0.670 & 0.824 &  4.58 & 6  \\
SCR2  &  30 &  5.11 & 16.4 & 10.90 & 0.311 & 0.330 & 0.332 & 0.632 & 0.803 & 21.52 & 7  \\
BSR1  & 300 &  3.53 & 18.7 &  9.36 & 0.403 & 0.370 & 0.370 & 0.617 & 0.787 & 30.31 & 4  \\
SCR4  &  88 &  4.06 & 19.9 & 10.98 & 0.325 & 0.355 & 0.352 & 0.584 & 0.766 &  7.25 & 5  \\
TAB1  &  81 &  4.35 & 17.4 & 11.25 & 0.304 & 0.309 & 0.311 & 0.576 & 0.761 & 10.05 & 8  \\
SCR3  &  30 &  5.13 & 15.6 & 10.49 & 0.312 & 0.338 & 0.339 & 0.570 & 0.765 & 12.73 & 10 \\
NEP1  &  82 &  3.04 & 19.3 & 11.28 & 0.257 & 0.265 & 0.247 & 0.450 & 0.680 & 18.40 & 1  \\
BUL12 &  29 &  4.70 & 15.4 &  9.73 & 0.273 & 0.425 & 0.422 & 0.437 & 0.679 & 44.40 & 1  \\
\hline
\textbf{Mean} &  & \textbf{5.45} & \textbf{20.3} & \textbf{11.51} & \textbf{0.329} & \textbf{0.339} & \textbf{0.338} & \textbf{0.662} & \textbf{0.816} & & \\
\hline
\end{tabular}
\end{table}

Across the 17 farms, the mean in-sample $R^2$ is 0.66 and the mean correlation between estimated and realised hourly capacity factors is 0.82, indicating that the sigmoid model captures the dominant wind speed--power relationship well. The fitted parameters $\alpha$ and $\beta$ are highly stable across semi-annual windows within each farm: the coefficient of variation of $\alpha$ ranges from 2.2\% to 5.9\%, and that of $\beta$ from 2.9\% to 15.0\%, confirming that the power curve is a stable physical characteristic of each site.

The variation in $R^2$ across farms is largely explained by the quality of the wind speed estimate, which depends on the proximity and density of surrounding weather stations.  As shown in Figures~\ref{fig:map_province} and~\ref{fig:map_zoom}, the ACIS weather station network provides dense coverage in southern Alberta but is sparse in the northeast. Farms such as CRR1, CR1, ARD1, and AKE1, which lie within a cluster of stations in the Pincher Creek area, achieve $\overline{R^2} \geq 0.76$. By contrast, BUL12 and NEP1, which are served by only a single station within 50km at distances of 44 and 18km respectively, yield $\overline{R^2}$ values of 0.44 and 0.45, reflecting degraded wind speed interpolation quality rather than a failure of the sigmoid model. Farms in central Alberta (SCR2, SCR3, SCR4, TAB1) show intermediate $\overline{R^2}$ values (0.57-0.63).
\begin{figure}[htbp]
    \centering
    \includegraphics[width=0.65\textwidth]{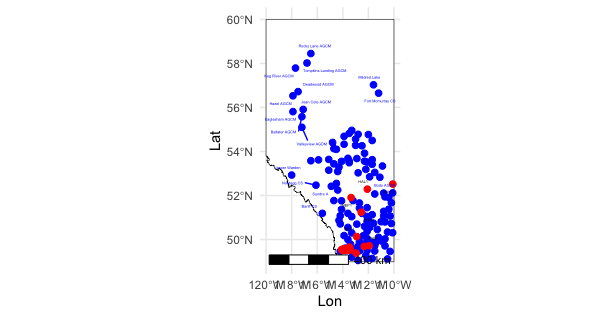}
    \caption{Location of wind farms (red circles) and Alberta Government
             weather stations (blue circles) across the province.
             Station density is substantially higher in southern Alberta,
             leading to more accurate hub-height wind speed interpolation
             for farms in that region.}
    \label{fig:map_province}
\end{figure}

\begin{figure}[htbp]
    \centering
    \includegraphics[width=0.90\textwidth]{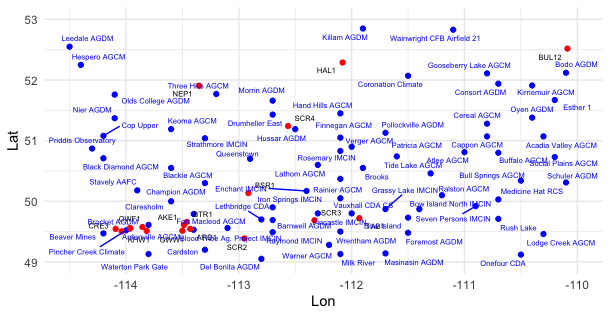}
    \caption{Zoomed view of southern Alberta showing the location of
             wind farms (red) and weather stations (blue). Farms in this
             region benefit from short interpolation distances and achieve
             the highest goodness-of-fit values in Table~\ref{tab:training_results}.
             Isolated farms such as BUL12 (far northeast) and NEP1 are not
             visible at this scale and rely on more distant station
             interpolation, corresponding to lower $\overline{R^2}$.}
    \label{fig:map_zoom}
\end{figure}

\subsubsection{Out-of-Sample Validation (2020-2023)}
\label{sec:oos_validation}

To validate the temporal transferability of the sigmoid model, we perform an out-of-sample test in which the averaged training-period parameters $(\bar\alpha_i, \bar\beta_i, \bar{v}_{\min,i}, \bar{v}_{\max,i})$ from 2016-2019 are applied to hourly wind speed data from 2020-2023 without re-fitting. For each farm $i$ and test-period hour $t$, the estimated capacity factor is computed as
\begin{equation}\widehat{\mathrm{ECF}}_{i,t} \;=\;
    \begin{cases}
        0, & v_t < \bar{v}_{\min,i}, \\[4pt]
        \sigma(v_t;\,\bar\alpha_i,\bar\beta_i),
        & \bar{v}_{\min,i} \leq v_t \leq \bar{v}_{\max,i}, \\[4pt]
        1, & v_t > \bar{v}_{\max,i},
    \end{cases}
    \label{eq:oos_predict}
\end{equation}
and compared against the realised capacity factor $\mathrm{ACF}_{i,t}$.

Table~\ref{tab:oos_validation} presents the out-of-sample results for all 17 farms. The overall mean hourly $R^2$ is 0.63 and the mean correlation is 0.81, closely tracking the in-sample values of 0.66 and 0.82 from Table~\ref{tab:training_results}. This consistency indicates that the sigmoid parameters estimated on 2016-2019 data generalise well to the subsequent four-year period, with no evidence of degradation over time.
\begin{table}[htbp!]
\centering
\begin{tabular}{lcccccc}
\hline
\textbf{Farm} & $R^2$ & $r$ & $\overline{\mathrm{ACF}}$ & $\overline{\mathrm{ECF}}$ & \textbf{Bias} & \textbf{CF\_APE\,(\%)} \\ \hline
AKE1  & 0.708 & 0.853 & 0.347 & 0.362 &  0.0145    & 4.18  \\
ARD1  & 0.712 & 0.857 & 0.377 & 0.383 &  0.0063    & 1.67  \\
BSR1  & 0.646 & 0.809 & 0.409 & 0.435 &  0.0262    & 6.42  \\
BTR1  & 0.683 & 0.830 & 0.359 & 0.354 & $-$0.0044  & 1.23  \\
BUL12 & 0.377 & 0.670 & 0.450 & 0.399 & $-$0.0501  & 11.1  \\
CR1   & 0.756 & 0.879 & 0.319 & 0.327 &  0.0084    & 2.62  \\
CRR1  & 0.754 & 0.874 & 0.346 & 0.362 &  0.0155    & 4.46  \\
GWW1  & 0.729 & 0.864 & 0.383 & 0.388 &  0.0048    & 1.24  \\
IEW1  & 0.626 & 0.805 & 0.350 & 0.352 &  0.0021    & 0.594 \\
IEW2  & 0.695 & 0.840 & 0.314 & 0.314 &  0.0003    & 0.081 \\
KHW1  & 0.714 & 0.853 & 0.375 & 0.387 &  0.0114    & 3.03  \\
NEP1  & 0.476 & 0.696 & 0.283 & 0.259 & $-$0.0245  & 8.65  \\
OWF1  & 0.716 & 0.858 & 0.367 & 0.387 &  0.0204    & 5.56  \\
SCR2  & 0.562 & 0.771 & 0.358 & 0.349 & $-$0.0100  & 2.78  \\
SCR3  & 0.554 & 0.758 & 0.361 & 0.363 &  0.0024    & 0.652 \\
SCR4  & 0.526 & 0.732 & 0.386 & 0.381 & $-$0.0051  & 1.33  \\
TAB1  & 0.538 & 0.748 & 0.323 & 0.340 &  0.0173    & 5.37  \\ \hline
\textbf{Mean} & \textbf{0.634} & \textbf{0.805} & \textbf{0.359} & \textbf{0.363} & \textbf{0.0009} & \textbf{3.58} \\ \hline
\end{tabular}
\caption{Out-of-sample validation results (2020--2023) using averaged sigmoid parameters from 2016--2019. $R^2$ and $r$: hourly-level coefficient of determination and correlation. $\overline{\mathrm{ACF}}$ and $\overline{\mathrm{ECF}}$: actual and estimated mean capacity factors averaged over all test-period hours. Bias~$= \overline{\mathrm{ECF}} - \overline{\mathrm{ACF}}$. CF\_APE: absolute percentage error of the mean capacity factor.}
\label{tab:oos_validation}
\end{table}

Since our objective is to obtain a reliable estimate of the \emph{expected} capacity factor at each location for the mean field game formulation, the most relevant validation metric is the accuracy of the mean capacity factor rather than the hourly prediction error. The mean bias across all farms is 0.0009 (effectively zero), and 14 of the 17 farms have $|\text{bias}| < 0.025$. The mean absolute percentage error (CF\_APE) of the expected capacity factor is 3.6\% overall and below 5\% for 12 of the 17 farms. 

A linear regression of $\overline{\mathrm{ECF}}$ on $\overline{\mathrm{ACF}}$ yields a slope of 0.90 ($p < 10^{-6}$), an insignificant intercept ($p = 0.41$), and $R^2 = 0.79$. Forcing the regression through the origin gives a slope of 1.005 with standard error 0.012, confirming near-perfect proportionality between estimated and actual expected capacity factors. Figure~\ref{fig:ACFvsECF} shows the plots for the out-of-sample mean estimated capacity 
factor against the mean actual capacity factor for each of the considered wind farms 
over the period 2020-2023.

The two farms with the weakest performance, BUL12 (CF\_APE = 11.1\%) and NEP1 (CF\_APE = 8.7\%), are the same farms that exhibited the lowest in-sample $R^2$ values. 
\begin{figure}[htbp]
    \centering
    \includegraphics[width=0.65\textwidth]{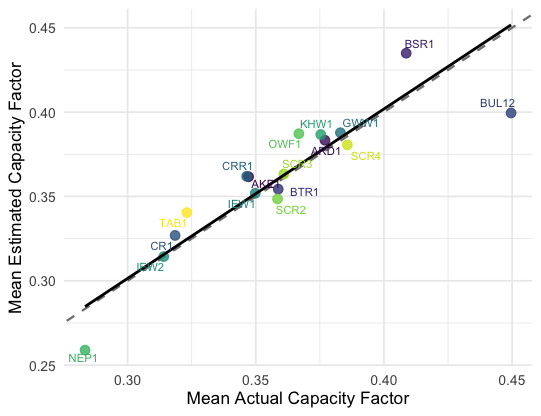}
    \caption{Out-of-sample mean estimated vs mean actual capacity factor 
         for each wind farm (2020-2023). The dashed line is
         ($y = x$); the solid line is the no-intercept regression 
         (slope $= 1.005$, $R^2 = 0.79$). Points cluster tightly around 
         the dashed line, confirming that the averaged sigmoid parameters 
         recover the expected capacity factor at each location.}
    \label{fig:ACFvsECF}
\end{figure}

Since agents in the MFG are assumed to be technologically identical, we require a single representative power curve. To do this, we restrict our attention to farms with mean in-sample $R^2 \geq 0.75$, which gives us a subset of four farms: CRR1, CR1, ARD1 and AKE1. The representative sigmoid parameters $\bar{\alpha},\bar{\beta},\bar{v}_{min}$ and $\bar{v}_{max}$ are then computed as the mean of the training-period estimates across this subset and are reported in Table~\ref{tab:rep_sigmoid}. We will use these estimates in our MFG.
\begin{table}[htbp!]
\centering
\begin{tabular}{cccc}
\hline
$\bar{\alpha}$ (m s$^{-1}$) & $\bar{\beta}$ & $\bar{v}_{\min}$ (m s$^{-1}$) & $\bar{v}_{\max}$ (m s$^{-1}$) \\
\hline
12.298 & 0.352 & 5.839 & 21.407 \\
\hline
\end{tabular}
\caption{Representative sigmoid power curve parameters computed as the mean over the four farms with mean in-sample $R^2 \geq 0.75$ (CRR1, CR1, ARD1, AKE1) across all training windows.}
\label{tab:rep_sigmoid}
\end{table}

Now that we empirically justified that our sigmoid approximation gives a reliable estimate for the average capacity factor, we use the ACIS weather station data \citep{ACIS2024} over our chosen hub height of $100\,m$ over a grid covering Alberta under various policy scenarios(see Section~\ref{sec:policy_AB}) to get the estimated average capacity factors visualized as heatmaps in Figure~\ref{fig:heatmaps_CF}. Here, each grid point represents a $35 \times 35$\,km$^2$ area.

\section{\label{sec:Covariance_Model} Covariance structure of windfarms in Alberta}
We model the covariance structure of standardized hourly wind power output across Alberta's windfarms from \citep{AESO_MeteredVolumes} using a two step approach.
\begin{enumerate}
    \item Principal Component Analysis (PCA) to capture large-scale spatial patterns.
    \item A parametric spatial residual model to account for localized correlations.
\end{enumerate}

Let $\mathbf{X}\in \mathbb{R}^{T\times N}$ be the matrix of standardized hourly power values, where $T$ is the number of hourly observations and $N$ is the number of wind farms. The PCA decomposition of the empirical covariance matrix is given by: $\boldsymbol{\widehat{C}_z} = \boldsymbol{\Psi} \boldsymbol{\Lambda} \boldsymbol{\Psi}'$ where 
\begin{itemize}
    \item $\boldsymbol{\Psi}=(\boldsymbol{\psi}_1,...,\boldsymbol{\psi}_N) \in\mathbb{R}^{N\times N}$ is the matrix of spatial eigenvectors or ``spatial modes". Each spatial model $\psi_k\in \mathbb{R}^N$ is defined over windfarm locations $X_1,...,X_N$ and is represented as: 
    \[
    \boldsymbol{\psi}_k=(\psi_k(X_1),...,\psi_k(X_N))^T \quad for\quad k=1,...,N
    \]
    \item $\boldsymbol{\Lambda}=diag(\lambda_1,...,\lambda_N)$ represents a diagonal matrix of non-negative eigenvalues, with each $\lambda_k$ representing the variance explained by the $k$-th spatial mode.
\end{itemize}

Examining the empirical data semi-annually from 2016 through 2018, we find that the first three Empirical Orthogonal functions (EOFs) successfully capture over 85$\%$ of the variation in the data.  We construct a rank 3 approximation of the covariance structure and analyze the residual correlations. Figure~\ref{fig:Residual_fit} shows the behavior of the correlation of these residuals against distance. We see a sharp initial decay in correlation and smooth oscillations that fades quickly. The residuals appear mostly uncorrelated and random, especially beyond 100 -150km. We can see that short range correlations exist and we need a model that captures this localized behavior. We propose a purely spatial dampened oscillation model of the form:
\begin{equation}\label{eq:covariance_kernel}
    C_X(h)=e^{-\lambda||h||}cos(\theta||h||)
\end{equation}
where, $\lambda$: Fitted Decay parameter, $\theta$: Fitted Oscillation frequency parameter and $h$: Distance between wind farms.

\begin{figure}[htbp]
    \centering
    \includegraphics[width=0.65\textwidth]{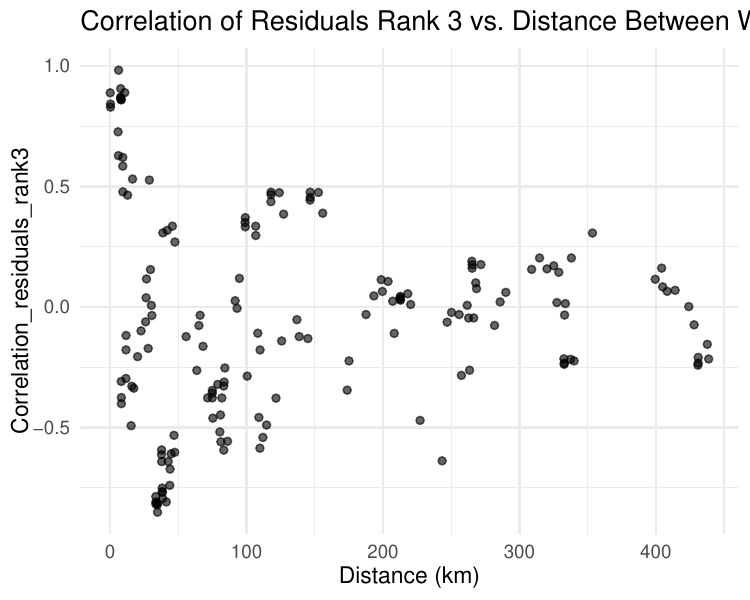}
    \caption{Residual correlation vs distance}
    \label{fig:Residual_fit}
\end{figure}

This form is not only empirically motivated but also theoretically justified: it satisfies the conditions of \textit{Bochner’s Theorem} \citep{Stein1999}, which guarantees that any continuous, positive-definite, and stationary covariance function corresponds to a valid spectral density. The product of an exponential decay and a cosine function yields a stationary covariance structure with a decaying, oscillatory pattern that remains positive definite making it suitable for spatial modeling of residual correlations.

We use the Weighted Least Squares (WLS) estimator to estimate the values of our fitted parameters. Recall WLS is defined as:
\begin{equation}\hat{\theta}_{WLS}=argmin\sum_h\sum_u\left(\frac{\hat{C}_X(h)-C_X(h|\Theta)}{1-C_X(h|\Theta)}\right)^2,
\end{equation}
where $\hat{C}_X(h)$ is the empirical correlation and $C_X(h|\Theta)$ is the fitted correlation. We see that this model captures the behaviour over the short range distances and we proceed to build the full covariance using rank 3 approximated covariance model and the proposed spatially dampened oscillation covariance model for residuals.

The variance of residuals can be found by:
\begin{equation}
    Var[R(X_m)]=1-\sum_{i=1}^3\psi_i^2(X_m)\lambda_i
\end{equation}
The covariance between windfarms located at $X_l$ and $X_r$ is given by:
\begin{equation}Cov(W(X_l),W(X_r))=\sum_{i=1}^3\psi_i(X_l)\psi_i(X_r)\lambda_i+Cov(R(X_l),R(X_r))
\end{equation}
\noindent Let $\boldsymbol{\Phi}^3 = (\boldsymbol{\psi}_1, \boldsymbol{\psi}_2, \boldsymbol{\psi}_3) \in \mathbb{R}^{N \times 3}$ be the matrix of the first three spatial principal components evaluated at $N$ wind farm locations, and let $\boldsymbol{\Lambda_3} = \mathrm{diag}(\lambda_1, \lambda_2, \lambda_3)$ be the diagonal matrix of associated eigenvalues.

Define $\mathbf{H} \in \mathbb{R}^{N \times N}$ as the distance matrix where $\mathbf{H}_{ij} = \|X_i - X_j\|$ is the Euclidean distance between wind farms $i$ and $j$.  Then, the full covariance matrix $\boldsymbol{C} \in \mathbb{R}^{N \times N}$ is given by:
\begin{equation}
    \boldsymbol{C} = \boldsymbol{\Phi}^3 \boldsymbol{\Lambda}_3 (\boldsymbol{\Phi}^3)^\top + \boldsymbol{R}
\end{equation}
where the residual covariance matrix $\boldsymbol{R}$ is defined element wise by:
\begin{equation}
    \boldsymbol{R}_{ij} = 
\exp(-\hat{\lambda} \mathbf{H}_{ij}) \cos(\hat{\theta} \mathbf{H}_{ij}) 
\cdot \sqrt{\text{Var}(R(X_i)) \cdot \text{Var}(R(X_j))}
\end{equation}

Note that we multiply by the residual standard deviations to convert correlation to covariance. We assess the quality of the fitted covariance structure by comparing it against the empirical covariance matrix derived from the original data. The fitted matrix combines the low-rank reconstruction from the first three principal components and a residual term modeled using a dampened oscillation kernel. 

We use the proposed covariance structure on semi-annual windows from years 2016-2018. Panel A (ACF vs. ECF) evaluates the model fit to sigmoid-estimated capacity factors. Panel B (ACF vs. ACF) evaluates the covariance model in isolation, fit directly to observed capacity factors. Pairwise metrics are computed on the off-diagonal entries of the covariance matrix. The \texttt{CovPerCap} metrics come from a linear regression of modeled versus empirical capacity-weighted row sums across farms. To evaluate the overall fit of the covariance structure, we compute the Frobenius norm of the error matrix and normalize it by the Frobenius norm of the empirical covariance matrix:
\[
\text{Relative Error} = \frac{\| \hat{\boldsymbol{C}} - \boldsymbol{C}_{\text{emp}} \|_F}{\| \boldsymbol{C}_{\text{emp}} \|_F}
\]
\begin{table}[htbp]
\centering
\caption{Validation of the modeled covariance against empirical covariance across six half-year windows.}
\label{tab:cov_validation}
\small
\setlength{\tabcolsep}{4pt}
\begin{tabular}{llcccccc}
\hline
\multirow{2}{*}{\textbf{Year}} & \multirow{2}{*}{\textbf{Period}} &
\multicolumn{4}{c}{\textbf{Pairwise covariance}} &
\multicolumn{2}{c}{\textbf{CovPerCap}} \\
\cline{3-6} \cline{7-8}
& & Cor & Frob. err & Slope & $R^2$ & Slope & $R^2$ \\
\hline
\multicolumn{8}{l}{\textit{Panel A: ACF vs. ECF (Panel A)}} \\
\hline
2016 & Jan--Jun & 0.953 & 0.162 & 0.935 & 0.953 & 1.068 & 0.929 \\
2016 & Jul--Dec & 0.974 & 0.161 & 0.950 & 0.948 & 1.085 & 0.907 \\
2017 & Jan--Jun & 0.961 & 0.126 & 0.947 & 0.961 & 1.019 & 0.958 \\
2017 & Jul--Dec & 0.945 & 0.139 & 1.024 & 0.946 & 1.071 & 0.938 \\
2018 & Jan--Jun & 0.975 & 0.154 & 0.926 & 0.950 & 0.970 & 0.922 \\
2018 & Jul--Dec & 0.973 & 0.142 & 1.030 & 0.948 & 1.123 & 0.915 \\
\hline
\multicolumn{8}{l}{\textit{Panel B: ACF vs. ACF (Panel B)}} \\
\hline
2016 & Jan--Jun & 0.991 & 0.064 & 1.008 & 0.983 & 0.987 & 0.997 \\
2016 & Jul--Dec & 0.990 & 0.066 & 1.007 & 0.980 & 0.973 & 0.993 \\
2017 & Jan--Jun & 0.993 & 0.056 & 0.999 & 0.985 & 0.979 & 0.997 \\
2017 & Jul--Dec & 0.990 & 0.066 & 1.007 & 0.980 & 0.979 & 0.997 \\
2018 & Jan--Jun & 0.992 & 0.062 & 1.007 & 0.984 & 0.987 & 0.998 \\
2018 & Jul--Dec & 0.993 & 0.052 & 1.003 & 0.986 & 0.981 & 0.998 \\
\hline
\end{tabular}
\end{table}

The results in Table~\ref{tab:cov_validation} demonstrate strong agreement between the modeled and empirical covariance structures across all six half-year validation windows. In Panel A, pairwise correlations range from 0.945 to 0.975, with Frobenius errors between 0.126 and 0.162, indicating that the sigmoid-estimated capacity factors introduce only modest additional 
error relative to the directly observed case. Panel B confirms this, showing near-perfect pairwise correlations (0.990--0.993) and Frobenius errors below 0.07 when the covariance model is fit directly to observed capacity factors. The CovPerCap slopes remain close to unity in both panels, with $R^2$ values consistently exceeding 0.90, confirming that 
the modeled covariance accurately reproduces the spatial dependence structure across farms. Taken together, these results validate the use of the dampened oscillation kernel~\eqref{eq:covariance_kernel} as a reliable 
approximation of spatial residual correlations among Alberta wind farms.

Having validated the proposed covariance model, we extend the analysis to a provincial grid covering Alberta, where each grid point represents a $35 \times 35$\,km$^2$ area. Under the various policy scenarios introduced by the Government of Alberta (see Section~\ref{sec:policy_AB}), the total covariance at each grid point is computed and visualized as heatmaps in Figure~\ref{fig:heatmaps_Cov}. Here it should be noted that the visualization is the aggregated quantity $\sum_{x'} Cov(W(x),W(x'))$, evaluated at every grid point $x$.

\section{Static MFG to Alberta}\label{sec:mfg_solution}

\begin{figure*}[htbp!]
    \centering
    \begin{subfigure}[t]{0.45\textwidth}
        \centering
        \includegraphics[width=\textwidth]{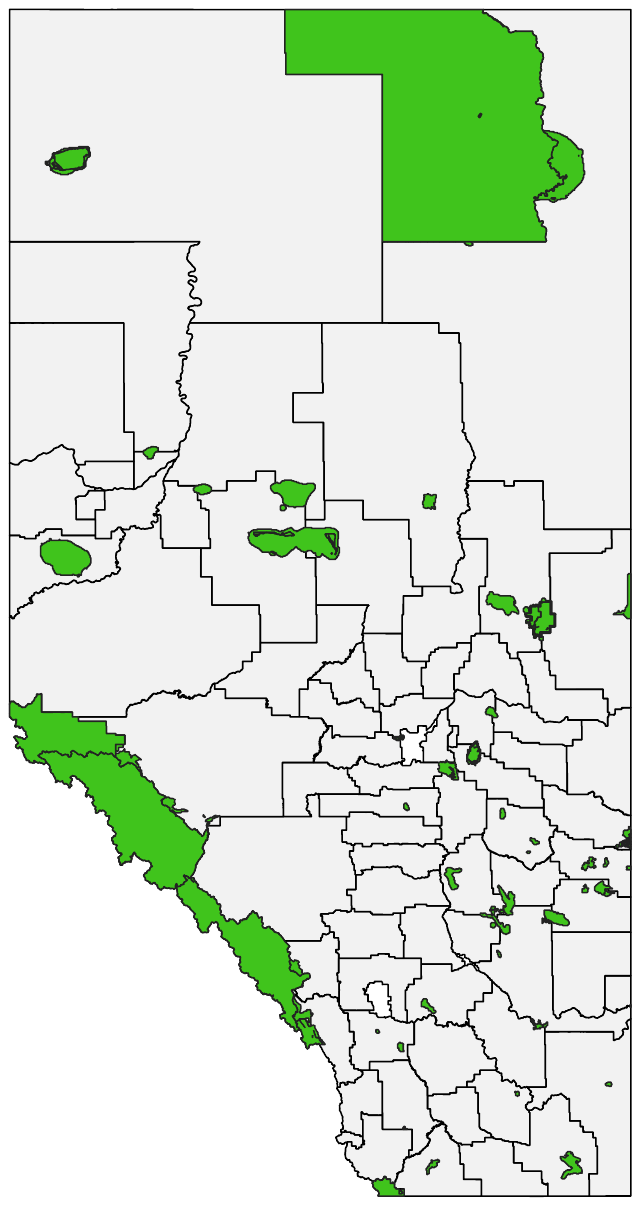}
        \caption{ESA AB}
        \label{fig:AB_ESA}
    \end{subfigure}
    \hfill
    \begin{subfigure}[t]{0.45\textwidth}
        \centering
        \includegraphics[width=\textwidth]{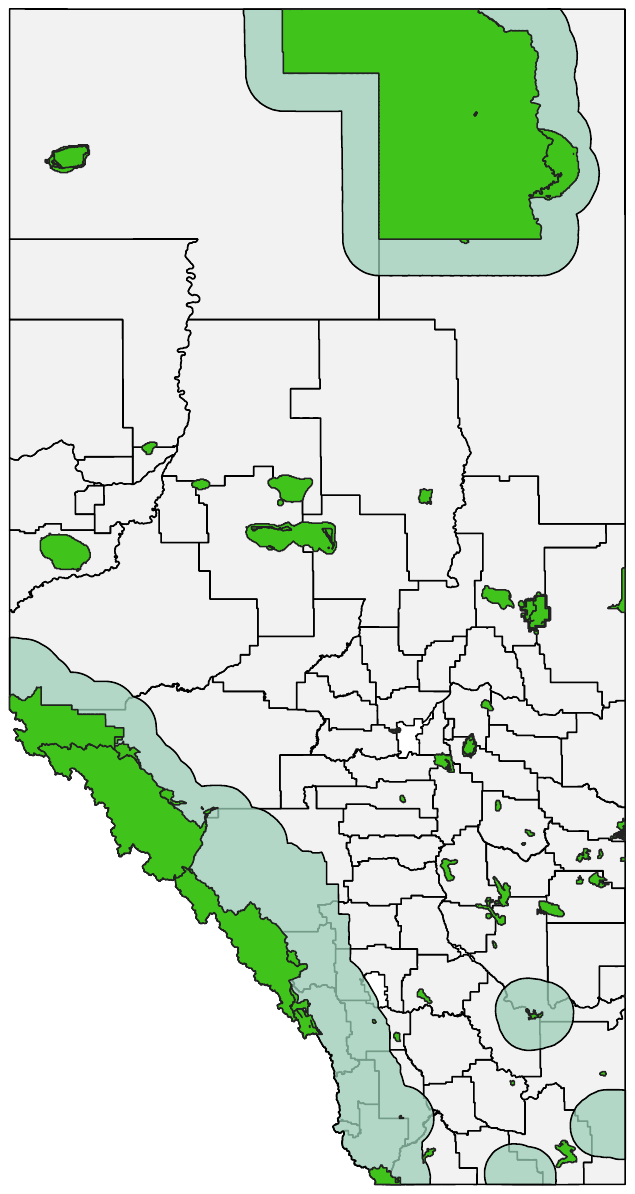}
        \caption{Viewscapes AB}
        \label{fig:AB_viewscapes}
    \end{subfigure}

    \vspace{0.5em}

    \begin{subfigure}[t]{0.45\textwidth}
        \centering
        \includegraphics[width=\textwidth]{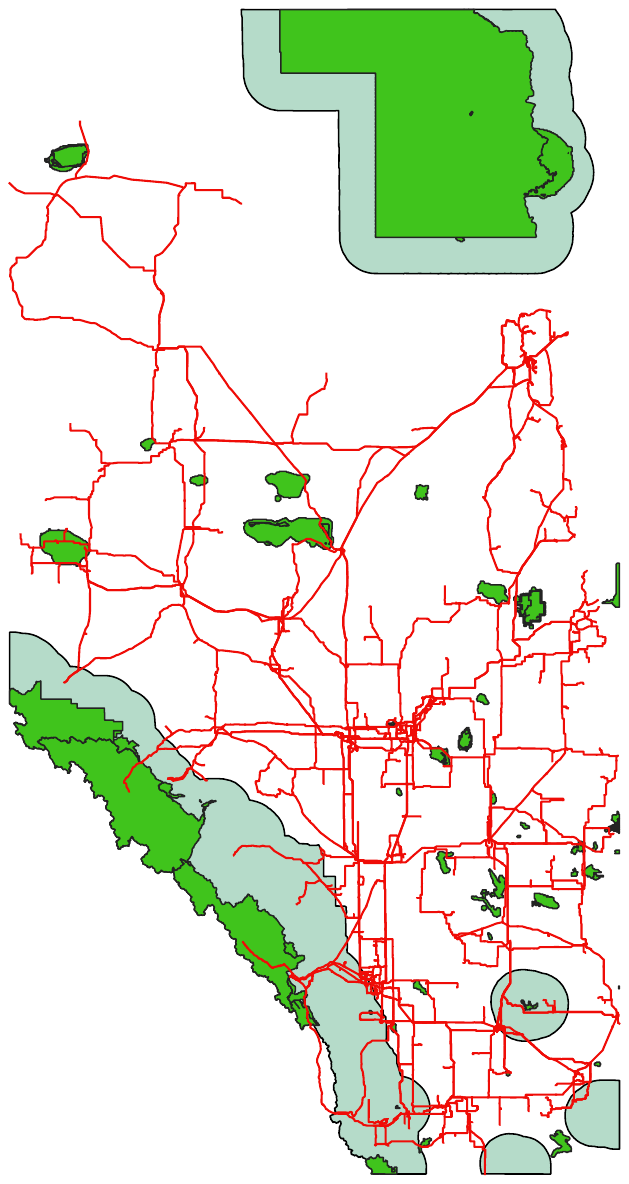}
        \caption{Transmission AB}
        \label{fig:AB_transmission}
    \end{subfigure}
    \caption{Policy scenarios across Alberta. (a)~Environmentally Sensitive Areas (green) excluded from development. (b)~Viewscape exclusion zones around designated scenic areas. (c)~Existing transmission infrastructure (red); under the transmission scenario, development is restricted to locations within 10\,km of these lines.}
    \label{fig:Policy_Scenarios}
\end{figure*}

\subsection{Policy Scenarios in Alberta}\label{sec:policy_AB}
Now we implement the static MFG theory discussed above to Alberta. Our goal is to compare the optimal capacity allocation and revenue distribution across four policy scenarios representing different combinations of short and long term constraints to development. The total land area of Alberta is 661,848\,km$^2$. We cover this area using grid points where each point covers a $35 \times 35$\,km$^2$ area. In the first scenario, an approximation of the rules pre 2024, we impose only minimal land restrictions, which exclude environmentally significant areas (ESAs) of international and national importance \citep{ESA2009} and the metropolitan areas of Calgary and Edmonton \citep{AESO_Municipalities}, yielding 1,434 grid points across Alberta. The second scenario incorporates the viewscape restrictions 
introduced by the Government of Alberta in February 2024 
\citep{Alberta_Viewscape2024} in addition to the minimal 
restrictions, with the spatial extent of restricted areas obtained from the AESO geospatial layer 
\citep{AESO_Viewscapes}, reducing the available grid to 
1228 points. The third scenario represents the pre-viewscape restriction transmission availability, restricting development to locations within 10\,km of existing transmission infrastructure \citep{AESO_Transmission}, which reduces the grid to 408 points. The fourth scenario combines both viewscape and transmission constraints on top of the minimal restrictions, yielding the most restrictive feasible set of 357 grid points. Figure~\ref{fig:Policy_Scenarios} shows the grid under each policy scenario discussed above.

We consider a capacity expansion goal of $\tilde{C} = 15,000$ MW in total. We enforce the density constraint that each area cannot contain more than $C_{max} = 300$ MW of capacity.


To find the MFG solution, we solve the quadratic programming problem
\begin{eqnarray}
    &&\max -\frac{1}{2}  \sigma \, \boldsymbol{w}^\top \mathbf{C} \, \boldsymbol{w} + \lambda \, \boldsymbol{E}[W]^\top \boldsymbol{w}\label{optimize2}
\end{eqnarray}
subject to
\[
\begin{aligned}
    w_1+...+w_{n} = 15000\\
    0 \leq w_i \leq 300 \text{  for $i=1,...n$},
\end{aligned}   
\] 
using the estimated capacity factors (Figure~~\ref{fig:heatmaps_CF}) and covariance (Figure\ref{fig:heatmaps_Cov}) with the estimated values of the parameters $\hat{\lambda}=8281$ and $\hat{\sigma}=-5.901$ to find the total capacity $w$ invested at each location under each policy scenario.

In all scenarios we consider, we check that the solution of the problem \eqref{optimize2} satisfies the non-negativity condition, namely $\tilde{F}(X^j,\tilde{\mu})\geq 0$ for all $X^j$ such that $w_j>0$ which verifies that $w$ correspond to the unique MFG solution with total capacity of $15000$ MW.
\begin{figure*}[htbp!]
    \centering
    \begin{subfigure}[t]{0.45\textwidth}
        \centering
        \includegraphics[width=\textwidth]{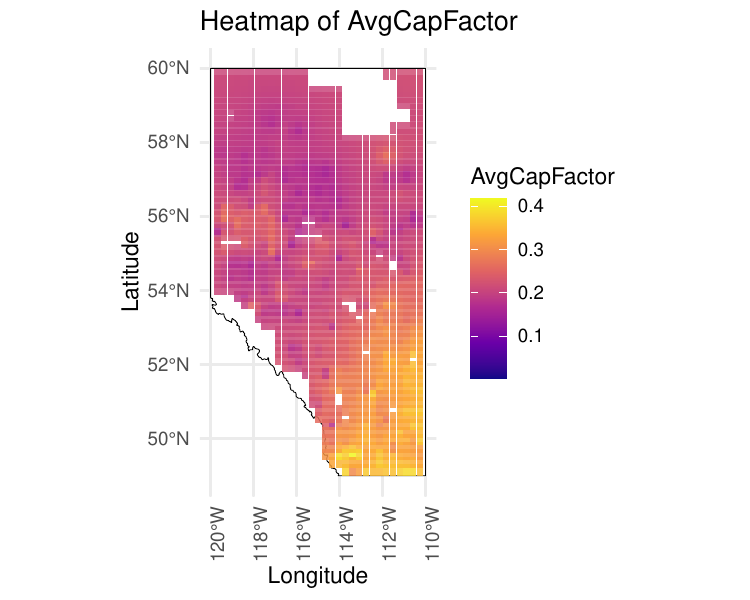}
        \caption{Minimal Restrictions}
        \label{fig:HM_CF_minimal}
    \end{subfigure}
    \hfill
    \begin{subfigure}[t]{0.45\textwidth}
        \centering
        \includegraphics[width=\textwidth]{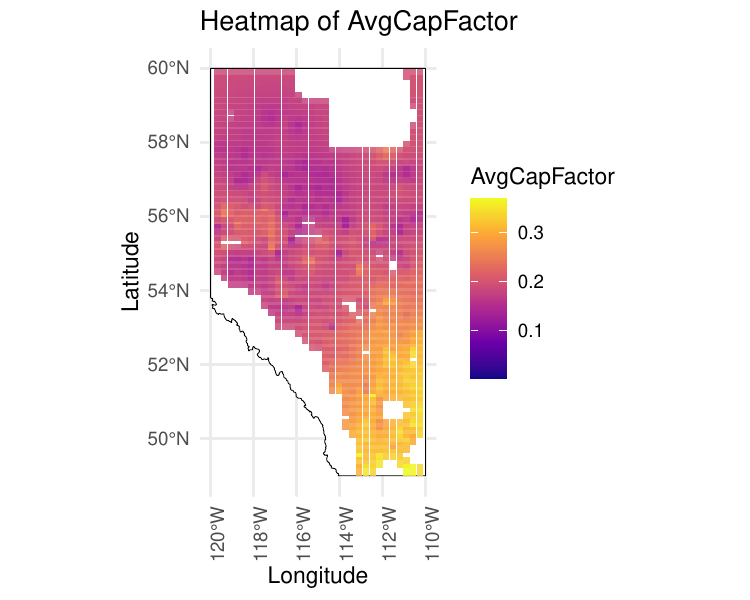}
        \caption{Viewscape Restrictions}
        \label{fig:HM_CF_viewscapes}
    \end{subfigure}

    \vspace{0.5em}

    \begin{subfigure}[t]{0.45\textwidth}
        \centering
        \includegraphics[width=\textwidth]{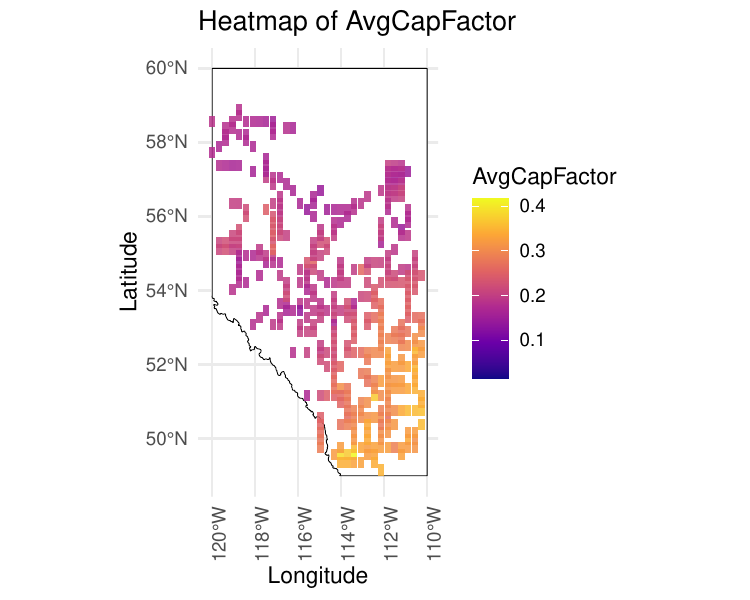}
        \caption{Transmission Restrictions}
        \label{fig:HM_CF_transmission}
    \end{subfigure}
    \hfill
    \begin{subfigure}[t]{0.45\textwidth}
        \centering
        \includegraphics[width=\textwidth]{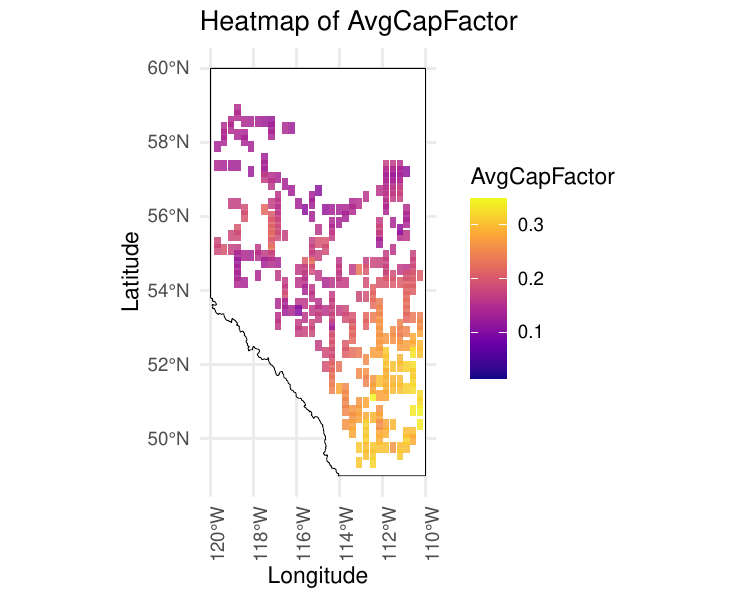}
        \caption{All Restrictions}
        \label{fig:HM_CF_all}
    \end{subfigure}

    \caption{Heatmaps of estimated average capacity factors across the four policy scenarios. Panels~(a)-(b) show the effect of viewscape exclusions on the feasible development area, while panels~(c)-(d) illustrate the additional spatial concentration imposed by the 10\,km transmission proximity constraint.}
    \label{fig:heatmaps_CF}
\end{figure*}

\begin{figure*}[htbp!]
    \centering
    \begin{subfigure}[t]{0.45\textwidth}
        \centering
        \includegraphics[width=\textwidth]{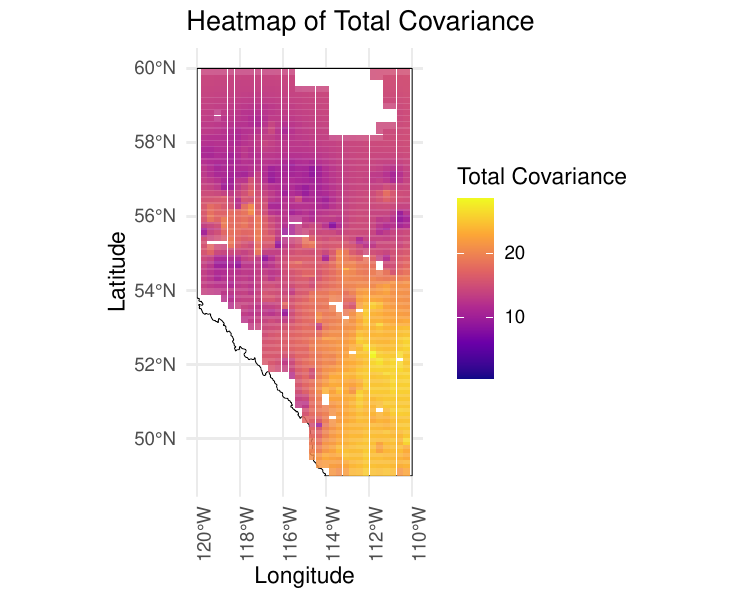}
        \caption{Minimal Restrictions}
        \label{fig:HM_Cov_minimal}
    \end{subfigure}
    \hfill
    \begin{subfigure}[t]{0.45\textwidth}
        \centering
        \includegraphics[width=\textwidth]{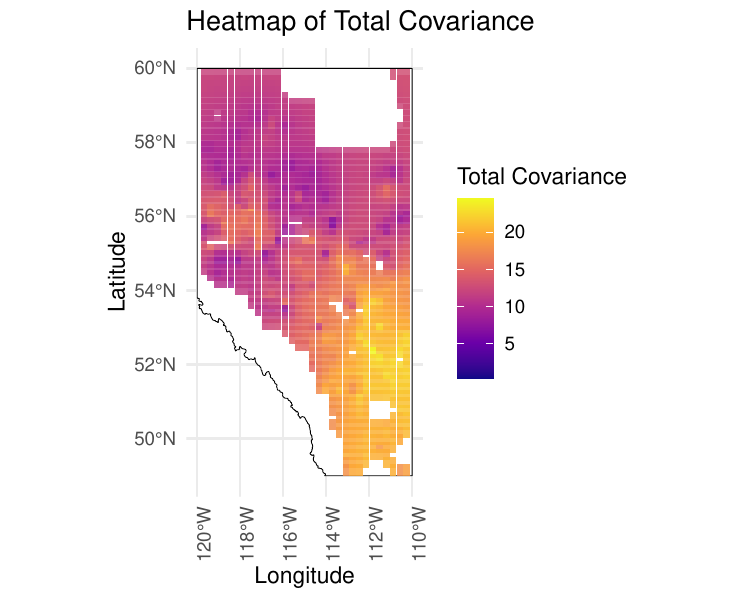}
        \caption{Viewscape Restrictions}
        \label{fig:HM_Cov_viewscapes}
    \end{subfigure}

    \vspace{0.5em}

    \begin{subfigure}[t]{0.45\textwidth}
        \centering
        \includegraphics[width=\textwidth]{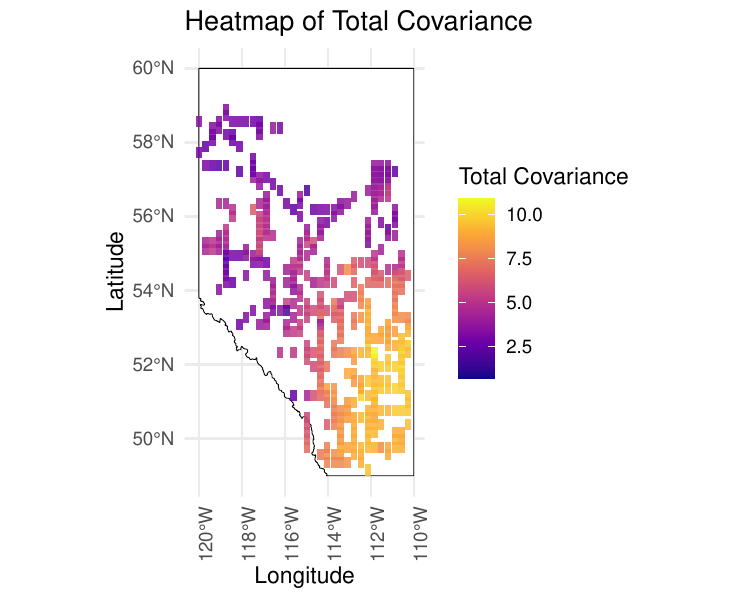}
        \caption{Transmission Restrictions}
        \label{fig:HM_Cov_transmission}
    \end{subfigure}
    \hfill
    \begin{subfigure}[t]{0.45\textwidth}
        \centering
        \includegraphics[width=\textwidth]{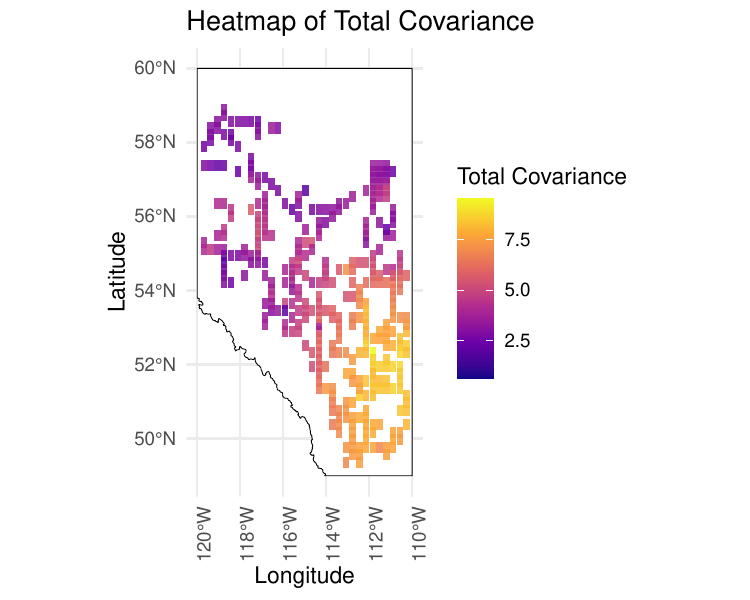}
        \caption{All Restrictions}
        \label{fig:HM_Cov_all}
    \end{subfigure}

    \caption{Heatmaps of total covariance per grid point $\sum_{x'} Cov(W(x),W(x'))$ across the four policy scenarios. Higher covariance indicates locations whose wind output is more correlated with neighbouring sites.}
    \label{fig:heatmaps_Cov}
\end{figure*}

\begin{figure*}[htbp!]
    \centering
    \begin{subfigure}[t]{0.45\textwidth}
        \centering
        \includegraphics[width=\textwidth]{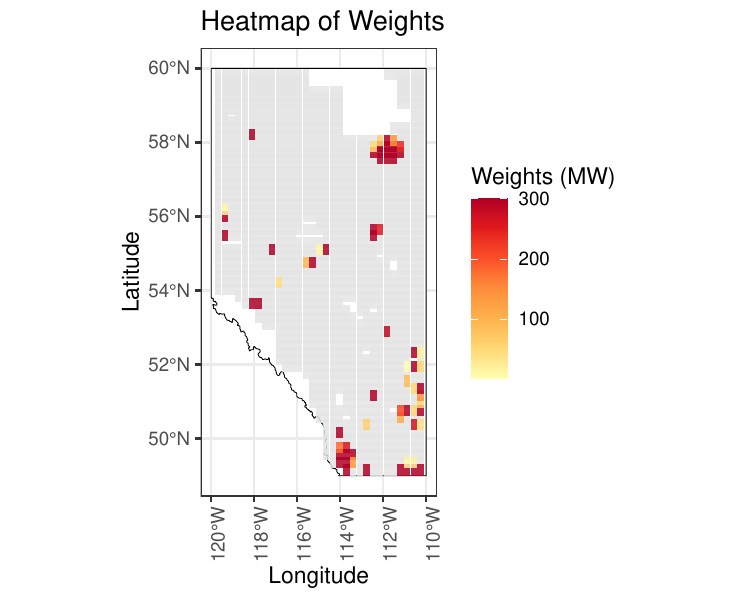}
        \caption{Minimal Restrictions}
        \label{fig:HM_weights_minimal}
    \end{subfigure}
    \hfill
    \begin{subfigure}[t]{0.45\textwidth}
        \centering
        \includegraphics[width=\textwidth]{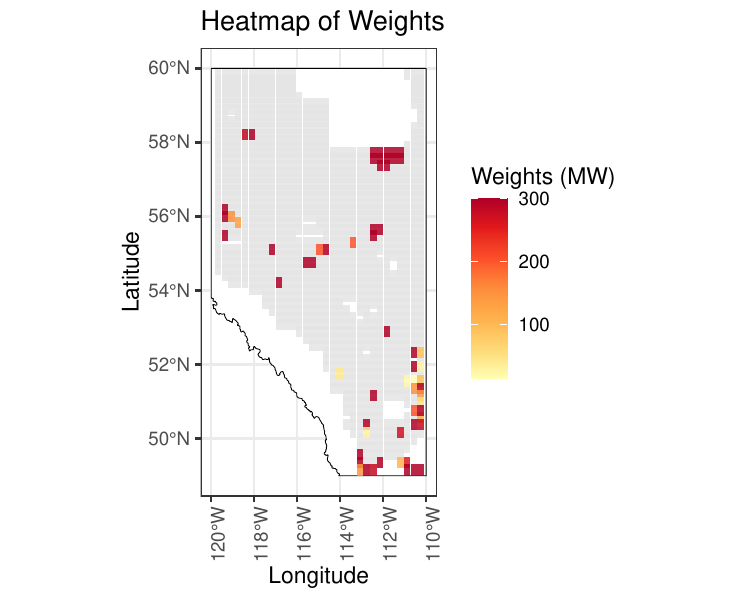}
        \caption{Viewscape Restrictions}
        \label{fig:HM_weights_viewscapes}
    \end{subfigure}

    \vspace{0.5em}

    \begin{subfigure}[t]{0.45\textwidth}
        \centering
        \includegraphics[width=\textwidth]{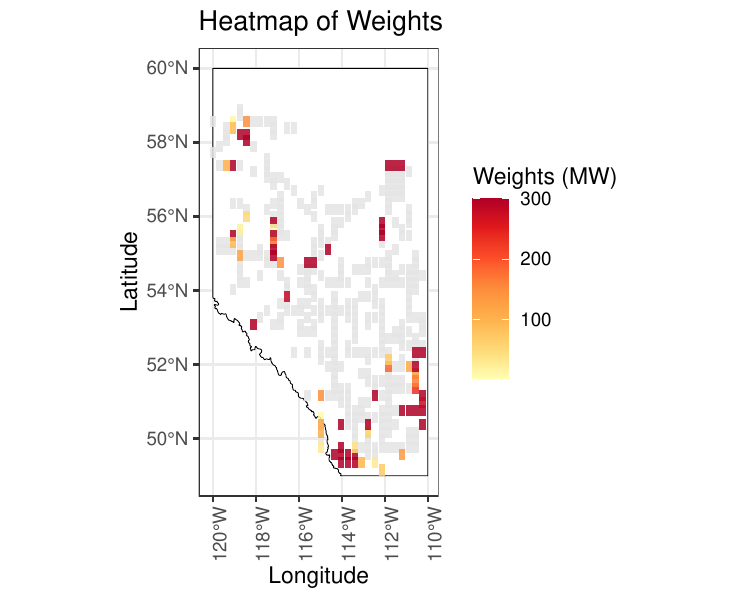}
        \caption{Transmission Restrictions}
        \label{fig:HM_weights_transmission}
    \end{subfigure}
    \hfill
    \begin{subfigure}[t]{0.45\textwidth}
        \centering
        \includegraphics[width=\textwidth]{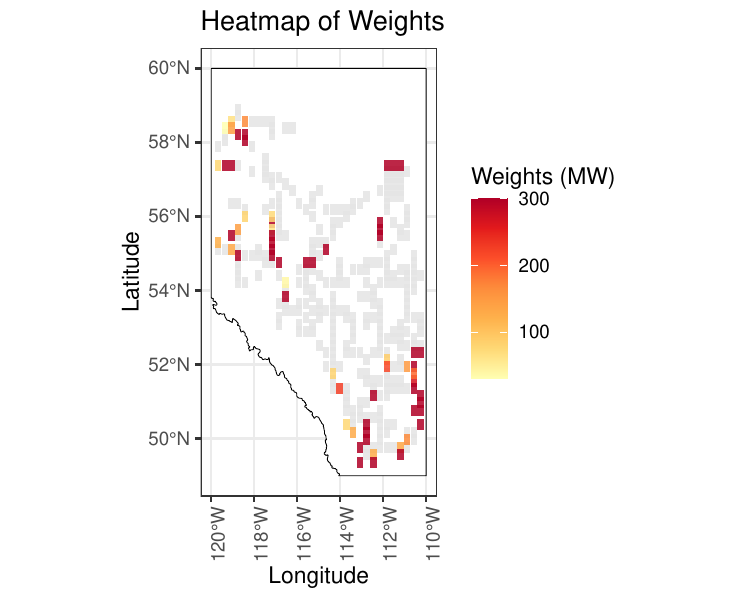}
        \caption{All Restrictions}
        \label{fig:HM_weights_all}
    \end{subfigure}

    \caption{Optimal wind farm capacity weights across the four policy scenarios. Grey areas indicate grid points with zero allocated capacity, while red shading indicates active sites with darker red corresponding to higher allocated capacity.}
    \label{fig:heatmaps_weights}
\end{figure*}

\begin{figure*}[htbp!]
    \centering
    \begin{subfigure}[t]{0.45\textwidth}
        \centering
        \includegraphics[width=\textwidth]{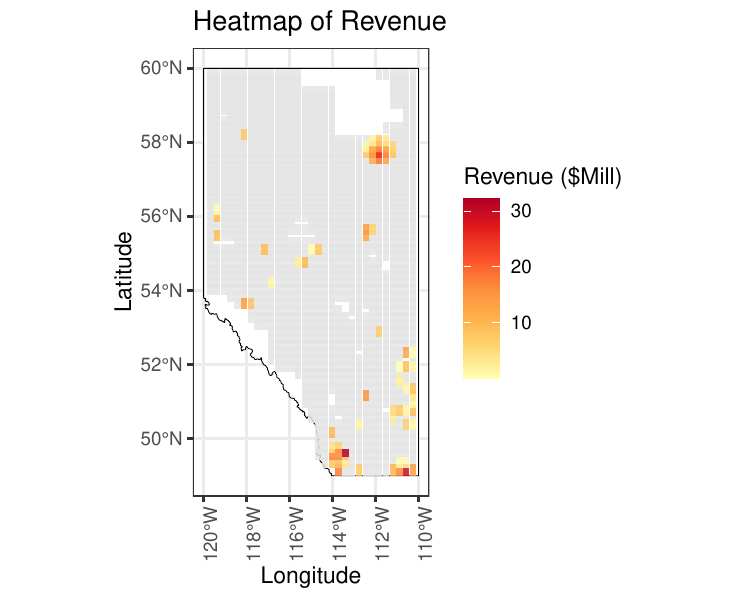}
        \caption{Minimal Restrictions}
        \label{fig:HM_revenue_minimal}
    \end{subfigure}
    \hfill
    \begin{subfigure}[t]{0.45\textwidth}
        \centering
        \includegraphics[width=\textwidth]{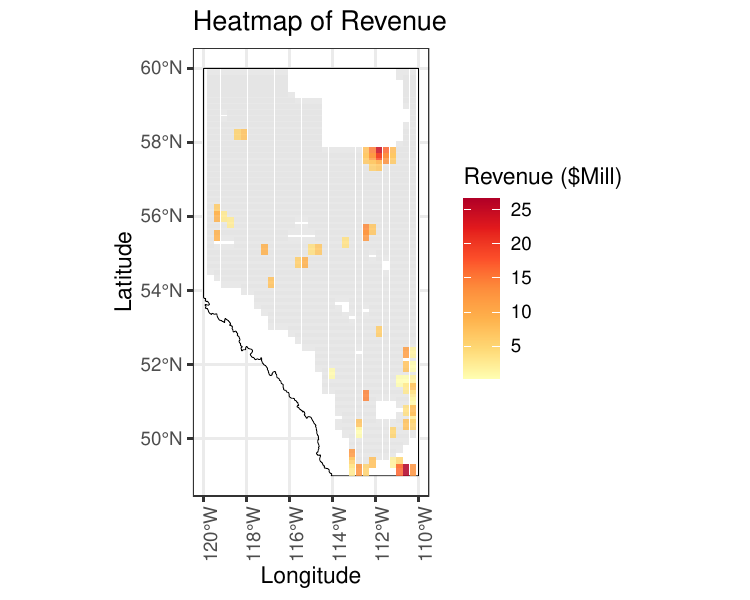}
        \caption{Viewscape Restrictions}
        \label{fig:HM_revenue_viewscapes}
    \end{subfigure}
    \vspace{0.5em}
    \begin{subfigure}[t]{0.45\textwidth}
        \centering
        \includegraphics[width=\textwidth]{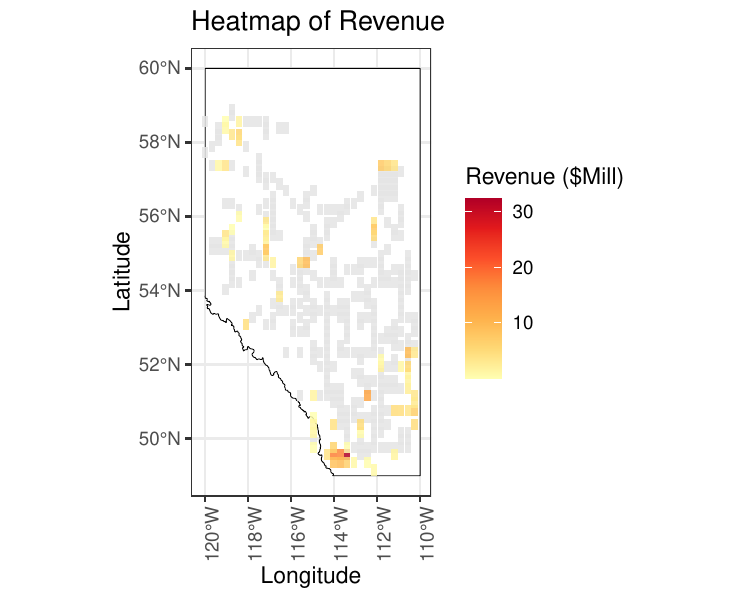}
        \caption{Transmission Restrictions}
        \label{fig:HM_revenue_transmission}
    \end{subfigure}
    \hfill
    \begin{subfigure}[t]{0.45\textwidth}
        \centering
        \includegraphics[width=\textwidth]{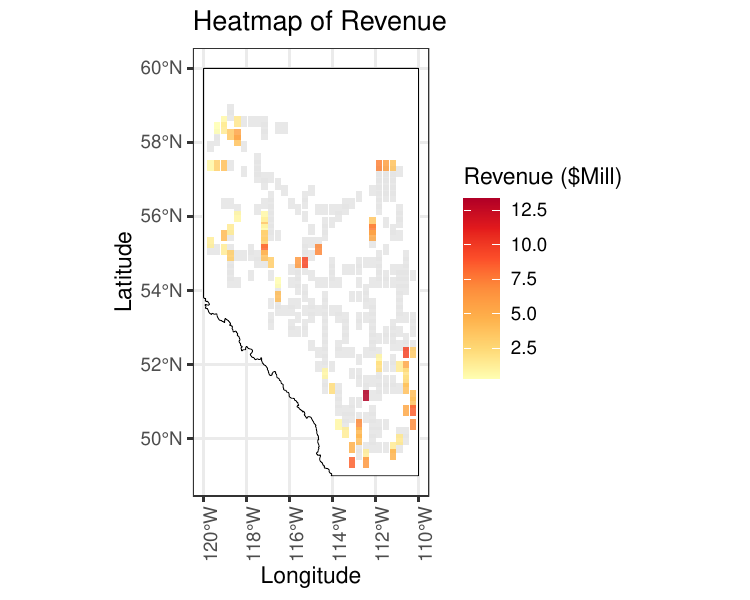}
        \caption{All Restrictions}
        \label{fig:HM_revenue_all}
    \end{subfigure}
    \caption{Spatial distribution of expected revenue across the four policy scenarios. Grey areas indicate grid points with zero allocated capacity, while colored sites show revenue with darker shading corresponding to higher expected revenue.}
    \label{fig:heatmaps_revenue}
\end{figure*}

A side-by-side comparison of the distribution of weights is depicted in Figure~\ref{fig:heatmaps_weights} and their associated revenue in Figure~\ref{fig:heatmaps_revenue}. In all scenarios, the total capacity installed is $\sum _{i = 1} ^ {n} \omega _{i} = 15,000$ MW, where $n$ is the total number of available areas on Alberta's grid under each policy scenario.

\begin{table*}[htbp!]
\centering
\begin{tabular}{lcccc}
\textbf{Metric} & \textbf{Minimal} & \textbf{Viewscape} & \textbf{Transmission} & \textbf{Land + Transmission} \\ \hline
Grid Points          & 1,434  & 1,228 ($\downarrow$14.4\%) & 408 ($\downarrow$71.5\%)   & 357 ($\downarrow$75.1\%)   \\
Active Points        & 68     & 62 ($\downarrow$7.4\%)     & 69 ($\uparrow$1.5\%)       & 64 ($\downarrow$5.9\%)     \\
Capped Points        & 40     & 39 ($\downarrow$2.5\%)     & 39 ($\downarrow$2.5\%)     & 39 ($\downarrow$2.5\%)     \\
Marginal Points      & 28     & 23 ($\downarrow$14.3\%)    & 30 ($\uparrow$7.1\%)       & 25 ($\downarrow$10.7\%)    \\
Total Capacity (MW)  & 15,000 & 15,000                     & 15,000                     & 15,000                     \\
Rev.\ per MW (\$)    & 39,637 & 31,388 ($\downarrow$20.7\%) & 20,223 ($\downarrow$49.0\%) & 16,322 ($\downarrow$58.8\%) \\
Total Revenue (\$M)  & 594.6  & 470.8 ($\downarrow$20.7\%) & 303.4 ($\downarrow$49.0\%) & 244.8 ($\downarrow$58.8\%) \\
Avg.\ Revenue (\$M)  & 8.74   & 7.59 ($\downarrow$14.4\%)  & 4.40 ($\downarrow$49.7\%)  & 3.83 ($\downarrow$56.2\%)  \\

\end{tabular}
\caption{\label{tab:policy_comparison}Comparison of energy and revenue metrics across four policy scenarios. ``Minimal'' imposes ESA and city spatial restrictions. ``Viewscape'' adds visual-impact exclusion zones on top of minimal. ``Transmission'' restricts development to locations within 10\,km of existing transmission infrastructure. ``Land\ + Transmission'' applies all land and transmission constraints simultaneously. Percentage differences are relative to the Minimal scenario.}
\end{table*}

\begin{table*}[htbp!]
\centering
\begin{tabular}{lccccc}
\textbf{Scenario} & \textbf{Total Active} & \textbf{Above \$5M} & \textbf{Below \$5M} & \textbf{\% Above} & \textbf{\% Below} \\ \hline
Minimal Restrictions       & 68 & 46 & 22 & 67.6\% & 32.4\% \\
Viewscape Restrictions     & 62 & 44 & 18 & 71.0\% & 29.0\% \\
Transmission Restrictions  & 69 & 20 & 49 & 29.0\% & 71.0\% \\
Land\ + Transmission      & 64 & 18 & 46 & 28.1\% & 71.9\% \\
\end{tabular}
\caption{\label{tab:viability}Number of active wind development sites above and below the \$5M CAD revenue viability threshold under each policy scenario. Percentages are calculated relative to the total number of active sites within each scenario.}
\end{table*}

In the resulting solutions, we highlight the locations that would have been developed under the earlier policy or without transmission availability constraints,  but are no longer feasible under the given restrictions. Figure~\ref{fig:lostrevenue} presents the geographic distribution of revenue losses associated with the viewscape land restrictions and transmission restrictions, while Table~\ref{tab:revenue_lost_both} details the corresponding sites and magnitudes. 

\begin{figure*}[htbp!]
    \centering
    \begin{subfigure}[t]{0.45\textwidth}
        \centering
        \includegraphics[width=\textwidth]{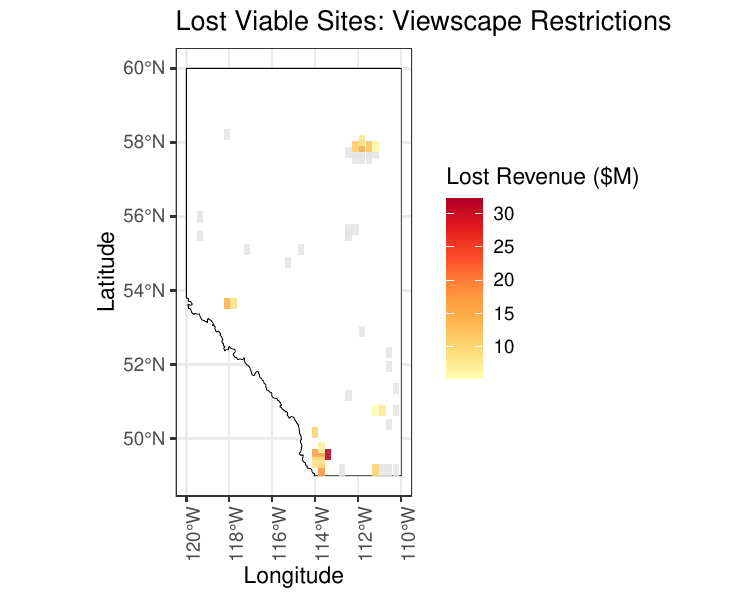}
        \caption{Viewscape Restrictions}
        \label{fig:lost_viewscapes}
    \end{subfigure}
    \hfill
    \begin{subfigure}[t]{0.45\textwidth}
        \centering
        \includegraphics[width=\textwidth]{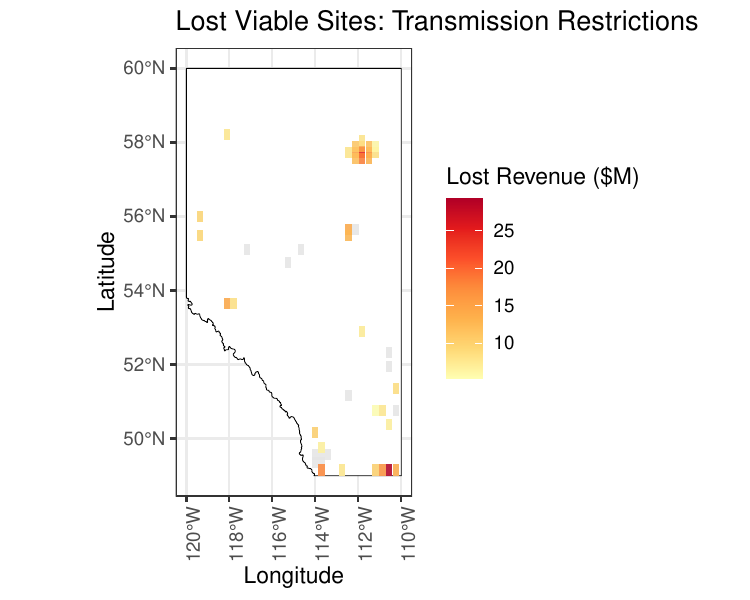}
        \caption{Transmission Restrictions}
        \label{fig:lost_transmission}
    \end{subfigure}

    \vspace{0.5em}

    \begin{subfigure}[t]{0.45\textwidth}
        \centering
        \includegraphics[width=\textwidth]{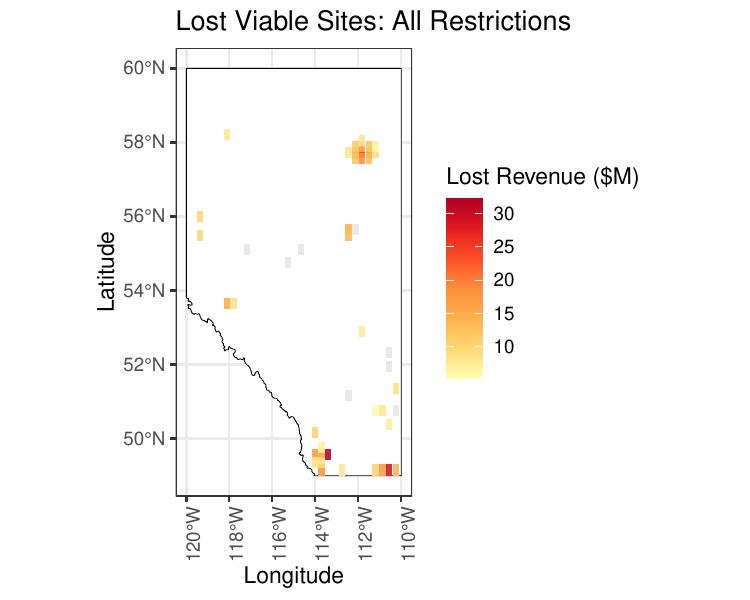}
        \caption{All Restrictions}
        \label{fig:lost_all}
    \end{subfigure}

    \caption{Geographic distribution of lost viable wind development sites relative to the minimal restrictions scenario. Grey tiles indicate retained viable sites; red shading indicates lost sites with darker red corresponding to higher lost revenue.}
    \label{fig:lostrevenue}
\end{figure*}

\begin{table}[htbp!]
\centering
\caption{Revenue and Capacity Lost at Different Locations Under Land \& Viewscape Restrictions (left) and Transmission Restrictions (right)}
\begin{minipage}[t]{0.48\textwidth}
\centering
\small
\begin{tabular}{|r|r|r|r|}
\hline
\textbf{Lon} & \textbf{Lat} & \textbf{Rev. (\$M)} & \textbf{Cap. (MW)} \\ \hline
-113.399 & 49.557 & 36.68 & 300 \\
-114.028 & 49.557 & 23.57 & 300 \\
-113.713 & 49.147 & 22.62 & 300 \\
-113.713 & 49.557 & 21.39 & 300 \\
-111.827 & 57.893 & 18.41 & 300 \\
-118.115 & 53.654 & 17.29 & 300 \\
-111.198 & 49.147 & 16.19 & 300 \\
-111.512 & 57.893 & 15.45 & 300 \\
-112.141 & 57.893 & 14.67 & 300 \\
-114.028 & 50.165 & 14.45 & 300 \\
-113.713 & 49.352 & 14.29 & 300 \\
-114.028 & 49.352 & 13.85 & 300 \\
-110.884 & 50.765 & 13.48 & 300 \\
-111.198 & 50.765 & 12.88 & 300 \\
-113.713 & 49.760 & 10.61 & 251 \\
-111.198 & 50.566 & 10.02 & 237 \\
-110.255 & 51.161 &  5.74 & 136 \\
-110.255 & 51.943 &  5.14 & 122 \\ \hline
\textbf{Total} & & \textbf{286.77} & \textbf{4,887} \\ \hline
\end{tabular}
\end{minipage}
\hfill
\begin{minipage}[t]{0.48\textwidth}
\centering
\small
\begin{tabular}{|r|r|r|r|}
\hline
\textbf{Lon} & \textbf{Lat} & \textbf{Rev. (\$M)} & \textbf{Cap. (MW)} \\ \hline
-113.399 & 49.557 & 36.68 & 300 \\
-110.569 & 49.147 & 35.11 & 300 \\
-111.827 & 57.725 & 29.70 & 300 \\
-114.028 & 49.557 & 23.57 & 300 \\
-113.713 & 49.147 & 22.62 & 300 \\
-111.827 & 57.557 & 22.44 & 300 \\
-110.884 & 49.147 & 22.38 & 300 \\
-113.713 & 49.557 & 21.39 & 300 \\
-111.512 & 57.725 & 21.04 & 300 \\
-110.255 & 49.147 & 19.67 & 300 \\
-112.456 & 55.654 & 19.62 & 300 \\
-112.141 & 57.725 & 19.13 & 300 \\
-111.827 & 57.893 & 18.41 & 300 \\
-112.456 & 55.476 & 17.55 & 300 \\
-118.115 & 53.654 & 17.29 & 300 \\
-111.512 & 57.557 & 16.68 & 300 \\
-111.198 & 49.147 & 16.19 & 300 \\
-112.141 & 57.557 & 15.64 & 300 \\
-111.512 & 57.893 & 15.45 & 300 \\
-119.373 & 56.007 & 15.31 & 300 \\
-112.141 & 57.893 & 14.67 & 300 \\
-112.770 & 49.147 & 14.55 & 300 \\
-114.028 & 50.165 & 14.45 & 300 \\
-116.857 & 54.209 & 14.43 & 300 \\
-110.569 & 50.366 & 14.35 & 300 \\
-113.713 & 49.352 & 14.29 & 300 \\
-114.028 & 49.352 & 13.85 & 300 \\
-110.884 & 50.765 & 13.48 & 300 \\
-113.399 & 55.298 & 13.25 & 300 \\
-111.198 & 50.765 & 12.88 & 300 \\
-113.713 & 49.760 & 10.61 & 251 \\
-111.198 & 50.566 & 10.02 & 237 \\
-119.373 & 55.476 &  8.96 & 212 \\
-110.884 & 49.352 &  6.34 & 150 \\
-110.255 & 51.943 &  5.14 & 122 \\ \hline
\textbf{Total} & & \textbf{596.94} & \textbf{9,921} \\ \hline
\end{tabular}
\end{minipage}
\label{tab:revenue_lost_both}
\end{table}

Table \ref{tab:lost_sites} gives the sites with transmission capability within 10km, but are impacted by the land use policies. These sites represent the short-term impact of viewscape regulations, reflecting the reduction in feasible wind development when visual impact rules are applied to the transmission-constrained solution. The spatial distribution of these lost sites is shown in Figure~\ref{fig:lost_sites}.

\begin{figure}[ht]
\centering
\includegraphics[width=0.8\textwidth]{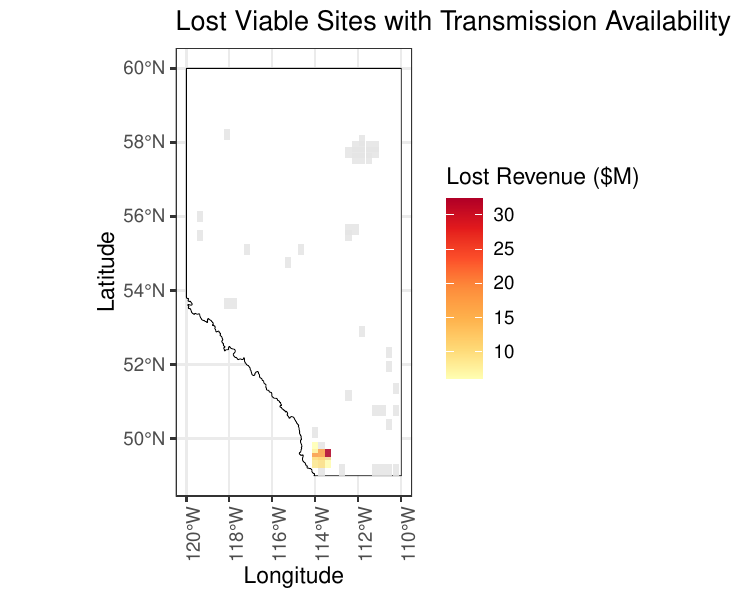}
\caption{Spatial distribution of viable wind development sites lost due to land restrictions under existing transmission constraints. Grey regions indicate retained viable sites, while colored sites represent lost capacity, shaded by revenue impact using a sequential yellow-orange-red scale. Site-level details are provided in Table~\ref{tab:lost_sites}.}
\label{fig:lost_sites}
\end{figure}

\begin{table}[ht]
\centering
\caption{Viable sites lost due to land restrictions under existing transmission constraints.}
\label{tab:lost_sites}
\begin{tabular}{cccc}
\hline
Latitude & Longitude & Capacity (MW) & Revenue (\$M) \\
\hline
49.557 & $-$113.399 & 300 & 32.36 \\
49.557 & $-$114.028 & 300 & 17.40 \\
49.557 & $-$113.713 & 300 & 16.80 \\
49.352 & $-$113.713 & 300 &  9.49 \\
49.352 & $-$114.028 & 300 &  8.58 \\
49.352 & $-$113.399 & 300 &  6.52 \\
49.760 & $-$114.028 & 300 &  6.14 \\
\hline
\multicolumn{3}{l}{\textbf{Total}} & \textbf{97.29} \\
\hline
\end{tabular}
\end{table}

The different scenarios explored here, which consider policy and infrastructure constraints, cover the short and long term opportunities for wind generation development. 

In the short term, before new transmission infrastructure could reasonably be built, the viewscape restrictions can be expected to be binding for the sites identified in Figure~\ref{fig:lost_sites}. These represent areas where transmission investments have already been made, but future development is not allowed.  

In the longer term, the viewscape restrictions have a more significant impact with additional viable areas in the north and east of the province off limits as showin in Figure~\ref{fig:lostrevenue}a.


Throughout the analysis we isolate only the physical impacts of the viewscape and transmission restrictions. As a result, the lost revenues calculated here are for the case where the total wind capacity still reaches 15,000 MW in all scenarios.  In doing so, we are explicitly excluding any impacts to wind development activity arising from policy risk assessed by renewable energy developers from the viewscape rules or the moratorium that proceeded it. The modeling is conducted using the existing market rules and excludes any changes such as increased transmission interconnection costs or other new rules proposed in the new Restructured Energy Market which are all out of scope for this analysis.

The static MFG theory is optimized assuming all locations are developed at the same time. In practice the build out of 15,000 MW of new capacity will take place over a number of years and revenues could be higher at a given site in the interim before correlated sites are developed. However, a given wind farm developer has no control over the decisions by future developers and it is therefore rational to plan their development assuming other projects may move forwards quickly. For centrally planned systems the MFG solution is particularly relevant for situations with significant rapid renewable development is planned, such as the one today where new wind generation is needed not just to meet growth but to displace higher emitting forms of generation.

The reward function we calculate is dependent on the wind covariance assuming the temporal shape of other supply as well as demand within the electricity grid remains largely fixed for times of wind production. Solar production is not likely to have a large impact given the relatively high degree of anti-correlation between the resources. However, significant future deployment of batteries and or demand side flexibility through load shifting (e.g., electric vehicle charging) could impact the reward function either positively or negatively. Future work could introduce adjustments to the historical data for determining the reward function that represent these potentially longer term changes to the market operations. But this longer term market uncertainty is not unique to the MFG solution.


The MFG approach demonstrated here represents a new approach to evaluating the economics of future wind generation development that could be applied to a number of challenges in electricity system planning. 

In the case of transmission planning, there is a clear need to increase access to transmission interconnection to develop new wind resources. Early efforts to enable wind generation through transmission build out such as the Competitive Renewable Energy Zones in Texas have focused on connecting areas with significant wind resources without consideration for covariance between sites. This approach is reasonable for early stages of wind development where limited amounts of wind generation make any covariance impacts small. However as the wind generation capacity increases it is more important to consider covariance to optimize costs and benefits. The geographical distribution of lost viable sites from transmission restrictions, Figure~\ref{fig:lostrevenue}b, shows not just where transmission is needed to access areas of significant wind resources, but the value of the resources, a key input to the transmission planning process.

Similarly, in the case of renewable energy auctions, there is an important challenge in enabling geographically optimal locations for development in a way that does not increase risk and therefore prices. The reward function could be used as a direct scaling function to renewable energy auction prices to represent the different value arising from co-variance. Unlike other solutions proposed such as benchmarks it would be fixed prior to the auction, thereby providing revenue certainty that does not increase financing costs above those in a geographically neutral auction format. 

Future work could integrate other variable resources such as solar, as well as evaluate the impact of supply and demand shifting through batteries or flexible demand.




\begin{figure}
  \centering
    \caption{}\label{fig1}
\end{figure}



\appendix

\section{Proofs}\label{app_proofs}

\subsection{Proof of Theorem~\ref{thm:epsilon_N-Nash}}

\begin{pf}
Proof of (i)
\newline
$[\tilde{\mu}_N(X^1),...,\tilde{\mu}_N(X^n)]$ is in a bounded domain since the total capacity at each location is bounded. Therefore, there exists a subsequence with a limit point such that:

$$[\tilde{\mu}_{N_m}(X^1),...,\tilde{\mu}_{N_m}(X^n)]\xrightarrow{} [\tilde{\mu}(X^1),...,\tilde{\mu}(X^n)]$$
and
$$[{\mu}_{N_m}(c^1,X^1),...,{\mu}_{N_m}(c^k,X^n)]\xrightarrow{} [{\mu}(c^1,X^1),...,{\mu}(c^k,X^n)]$$
\noindent
The limit $\mu$ must be a probability measure on $A$, and it satisfies the constraint~\eqref{eq:density_constraint}.

Proof of (ii)

Let $\delta=\min(\tilde{c}, \frac{\tilde{C}-\tilde{c}}{k})$ and $A_N:\{x :\tilde{\mu}_N(x) \geq C_{max} -\frac{\delta}{N}\}$.  Consider $x\in X_{max}$, i.e $\tilde{\mu}(x)=C_{max}$ and $c$ such that $\mu(c,x)>0$. Then $x\in A_{N_m}$ for all sufficiently large $m$, since $\tilde{\mu}_{N_m}\xrightarrow{}C_{max}$.

Let $\mu_{N}^i[a]$ be the probability measure obtained from $\mu_{N}$ by replacing the action of the $i^{th}$ agent with another action $a\in A$. We say that $ \mu_{N}^i[a]$ is feasible if it satisfies the density constraint ~\eqref{eq:density_constraint}.

Now, let $y$ be such that $\tilde{\mu}(y) < C_{\max}$. Then for all sufficiently large $m$, $y \in A_{N_m}^c$ since $\tilde{\mu}_{N_m}(y) \to \tilde{\mu}(y) < C_{\max}$. Furthermore, for sufficiently large $m$, $\mu_{N_m}^i[c', y]$ is feasible for all $c'\in \{c^1,\ldots c^k\}$ and for all $i$ with $X_i^{N_m} = x$.  Since $\mu_N$ is an $\epsilon_N$-Nash equilibrium, we have
\begin{eqnarray}\label{eNash}
\frac{1}{N_m}\sum_{i=1}^{N_m}C_i^{N_m} \tilde{F}(X_i^{N_m},\tilde{\mu}_{N_m})1_{\{X_i^{N_m}=x, C_i^{N_m}=c\}}\geq \frac{1}{N_m}\sum_{i=1}^{N_m}(c' \tilde{F}(y,\tilde{\mu}_{N_m}^i[c',y]-\epsilon_{N_m})1_{\{X_i^{N_m}=x,C_i^{N_m}=c\}},
\end{eqnarray}
where we $(C_i^N,X_i^N)$ denote the action of the $i^{th}$ agent for the $N$-player game.
Taking $m \to \infty$,  $\epsilon_{N_m} \to 0$ and using continuity of $F$ as $\mu^{N_m} \to \mu$:
\begin{equation*}
    F((c, x), \mu) \geq F((c', y), \mu), 
\end{equation*}
where we used $\mu(c,x)>0$ to deduce that $\lim_{m\rightarrow\infty}\frac{1}{N_m}\sum_{i=1}^{N_m}1_{\{X_i^{N_m}=x,C_i^{N_m}=c\}}=\mu(c,x)>0$.

Similarly, we can argue that for sufficiently large $m$, $\mu^i_{N_m}[c',x]$ is feasible and inequality \eqref{eNash} holds with $\mu^i_{N_m}[c',x]$ on the right side,  and letting $m\rightarrow\infty$ gives
\begin{equation*}
    F((c, x), \mu) \geq F((c', x), \mu), 
\end{equation*}
for $c'<c$.

Proof of (iii)

Let $x \notin X_{\max}$, so $\tilde{\mu}(x) < C_{\max}$, 
and let $c$ be such that $\mu(c, x) > 0$.  Let $i$ be an agent such that $X_i^{N_m} = x$ and $C_i^{N_m}=c$.

Again for any $y \notin X_{max}$, 
for all sufficiently large $m$ we have $y \in A_{N_m}^c$, 
meaning $\mu_{N_m}^i[c', y]$ is feasible for any 
$c' \in \mathcal{C}$. 
$X_i^{N_m} = x$.

The rest of the argument is similar to part (i).  We first establish inequality \eqref{eNash} for any $(c',y)$ with $y\notin A_{N_m}$ and we let $m\rightarrow\infty$ to get
\begin{equation*}
    F((c,x),\mu) \geq F((c',y'),\mu) 
    \quad \forall\,(c',y') \text{ with } 
    \tilde{\mu}(y') < C_{\max}.
\end{equation*}
Since this holds for all such $(c',y')$:
\begin{equation*}
    F((c,x),\mu) \geq 
    \sup_{[c' \in \mathcal{C},\, y' \notin X_{\max}]} 
    F((c',y'),\mu).
\end{equation*}
Also, trivially,
\begin{equation*}
    F((c,x),\mu) \leq 
    \sup_{[c' \in \mathcal{C},\, y' \notin X_{\max}]} 
    F((c',y'),\mu).
\end{equation*}
Combining both directions:
\begin{equation*}
    F((c,x),\mu) = 
    \sup_{[c' \in \mathcal{C},\, y' \notin X_{\max}]} 
    F((c',y'),\mu). \qed
\end{equation*}
\end{pf}

\subsection{Proof of Theorem~\ref{thm:MFE}}
\vspace{.2in}
\begin{pf}
Let $X_{max}=\{x :\tilde{\mu}(x) = C_{max}\}$.  Then
\begin{eqnarray}
\sum_{x\in X_{max}} \tilde{F}(x,\tilde{\mu})\tilde{\mu}(x)-\sum_{x\in X_{max}} \tilde{F}(x,\tilde{\mu})\tilde{m}(x)&\geq& \inf_{x\in X_{max}}(\tilde{F}(x,\tilde{\mu})) \sum_{x\in X_{max}} (\tilde{\mu}(x)-\tilde{m}(x) )\label{ineq1}
\end{eqnarray}
since $\tilde{m}(x)\leq C_{max}=\tilde{\mu}(x)$ for $x\in X_{max}$.
On the other hand,
\begin{eqnarray*}
\sum_{x\in X_{max}^c} \tilde{F}(x,\tilde{\mu})\tilde{\mu}(x)&=& \sup_{x\in X_{max}^c} \tilde{F}(x,\tilde{\mu}) \sum_{x\in X_{max}^c} \tilde{\mu}(x),
\end{eqnarray*}
and 
\begin{eqnarray*}
\sum_{x\in X_{max}^c} \tilde{F}(x,\tilde{\mu})\tilde{m}(x)&\leq& \sup_{x\in X_{max}^c} \tilde{F}(x,\tilde{\mu}) \sum_{x\in X_{max}^c} \tilde{m}(x)
\end{eqnarray*}
which means
\begin{eqnarray}
\sum_{x\in X_{max}^c} \tilde{F}(x,\tilde{\mu})\tilde{\mu}(x)-\sum_{x\in X_{max}^c} \tilde{F}(x,\tilde{\mu})\tilde{m}(x)&\geq& \sup_{x\in X_{max}^c} \tilde{F}(x,\tilde{\mu}) \sum_{x\in X_{max}^c} (\tilde{\mu}(x)- \tilde{m}(x))
\end{eqnarray}
Adding up the two inequalties 
\begin{eqnarray*}
\sum_{x} \tilde{F}(x,\tilde{\mu})\tilde{\mu}(x)-\sum_{x} \tilde{F}(x,\tilde{\mu})\tilde{m}(x)\\
&\geq& \inf_{x\in X_{max}}(\tilde{F}(x,\tilde{\mu}) \sum_{x\in X_{max}} (\tilde{\mu}(x)-\tilde{m}(x))-\sup_{x\in X_{max}^c}(\tilde{F}(x,\tilde{\mu}))\sum_{x\in X_{max}^c} (\tilde{m}(x)-\tilde{\mu}(x))\\
&=&( \inf_{x\in X_{max}}(\tilde{F}(x,\tilde{\mu})-\sup_{x\in X_{max}^c }(\tilde{F}(x,\tilde{\mu}))\sum_{x\in X_{max}} (\tilde{\mu}(x)-\tilde{m}(x))\\
&\geq& 0,
\end{eqnarray*}
since $0\leq \sum_{x\in X_{max}} (\tilde{\mu}(x)-\tilde{m}(x))= \sum_{x\in X_{max}^c} (\tilde{m}(x)-\tilde{\mu}(x))$ and $\inf_{x\in X_{max}}\tilde{F}(x,\tilde{\mu})-\sup_{x\in X_{max}^c }\tilde{F}(x,\tilde{\mu})>0$ since $\mu$ is an MFG solution.

\end{pf}

\subsection{Proof of Lemma~\ref{thm:Farkas_Lemma}}

\begin{pf}

Let 
\[
D_1 = \left\{ \mathbf{m} \,\middle|\, m_1  + \cdots + m_{kn} = 1 \text{ , } m_i \geq 0,\ i = 1, 2, \ldots, kn \right\}
\]
Let 
\[
D_2 = \left\{ \mathbf{w} \,\middle|\, c^1 \leq w_1 + w_2 + \dots + w_n \leq c^k\text{ ,  } 0 \leq w_i \leq c_k ,\ i = 1, \ldots, n \right\}
\]

\noindent Let $A \in \mathbb{R}^{(n+1) \times (kn)}$ and $b \in \mathbb{R}^{n+1}$ defined as follows.

\[
A = \begin{pmatrix}
c_1 & 0 & \dots & 0 & c_2 & 0 & \dots & 0 & c_3 & 0 & \dots & 0 & c_k & 0 & \dots & 0 \\
0 & c_1 & \dots & 0 & 0 & c_2 & \dots & 0 & \dots & c_3 & 0 & 0 & \dots & c_k & 0 & 0 \\
\vdots & \vdots & \vdots & \vdots & \vdots & \vdots & \vdots & \vdots & \vdots & \vdots & \vdots & \vdots & \vdots & \vdots & \vdots & \vdots \\
0 & 0 & \dots & c_1 & 0 & 0 & \dots & c_2 & 0 & 0 & \dots & c_3 & 0 & \dots & \dots & c_k \\
1 & 1 & \dots & 1 & 1 & 1 & \dots & 1 & 1 & 1 & \dots & 1 & 1 & \dots & \dots & 1
\end{pmatrix}
\]

\[
\begin{array}{c@{\hskip 2cm}c}
\mathbf{m} = 
\begin{pmatrix}
m_1 \\ m_2 \\ \vdots \\ m_{kn}
\end{pmatrix}
&
\mathbf{b} = 
\begin{pmatrix}
w_1 \\ w_2 \\ \vdots \\ w_n \\ 1
\end{pmatrix}
\end{array}
\]

\vspace{0.1cm}
Then by Farkas Lemma exactly one of the following two assertions is true:
\begin{enumerate}
    \item There exists   $\mathbf{m} \in \mathbb{R}^{kn}$  such that $A\mathbf{m} = \mathbf{b}$ and $\mathbf{m}\geq 0$,
    \item There exists $y \in \mathbb{R}^{n+1}$ such that $A^T y \geq 0$ and $\mathbf{b}^T y < 0$.
\end{enumerate}

We note that equivalence of $D_1$ and $D_2$ is equivalent to showing that the first assertion above is true, hence 
we need to show that the second condition is not true.


Consider $A^\top y \geq 0$, this implies:

\[
\begin{pmatrix}
c_1 & 0 & 0 & \cdots & 1 \\
0 & c_1 & 0 & \cdots & 1 \\
0 & 0 & c_1& \cdots & 1 \\
\vdots & \vdots & \vdots & \vdots & \vdots \\
0 & 0 &\cdots& c_1 & 1 \\
c_2 & 0 & 0 & \cdots & 1 \\
0 & c_2 & 0 & \cdots & 1 \\
\vdots & \vdots & \vdots & \vdots & \vdots\\
c_k & 0 & 0 & \cdots & 1\\
\vdots & \vdots & \vdots &\vdots& \vdots \\
0 & 0 & \cdots & c_k & 1
\end{pmatrix}
\begin{pmatrix}
y_1 \\ y_2 \\ y_3 \\ \vdots\\ y_{n+1}
\end{pmatrix}
=
\begin{pmatrix}
c_1 y_1 + y_{n+1} \\
c_1 y_2 + y_{n+1} \\
c_1 y_3 + y_{n+1} \\
\vdots \\
c_1 y_{n} + y_{n+1}\\
c_2 y_1 + y_{n+1} \\
c_2 y_2 + y_{n+1} \\
\vdots \\
c_k y_{1} + y_{n+1} \\
\vdots \\
c_k y_{n} + y_{n+1}
\end{pmatrix}
\geq
\begin{pmatrix}
0 \\ 0 \\ 0 \\ 0 \\ 0 \\ \vdots \\ 0
\end{pmatrix}
\]

We need to show that $\mathbf{b}^\top y < 0$ is not true.
We compute:

\[
\mathbf{b}^\top y = 
\begin{pmatrix}
w_1 & w_2 & w_3 & \cdots  & w_{n} & 1
\end{pmatrix}
\begin{pmatrix}
y_1 \\ y_2 \\ y_3 \\ \cdots \\ y_{n+1}
\end{pmatrix}
= \sum_{i=1}^{n} w_i y_i  + y_{n+1}
\]

\textbf{Consider two cases:}
\vspace{0.2cm}

\textbf{\underline{Case 1:}} $y_{n+1} \geq 0$

\vspace{0.1cm}
To satisfy $A^\top y \geq 0$:
\[
c_k y_1 + y_{n+1} \geq 0 \quad \Rightarrow \quad y_1 \geq -\frac{y_{n+1}}{c_k}
\]

Similarly,
\[
y_2 \geq -\frac{y_{n+1}}{c_k}, \quad y_3 \geq -\frac{y_{n+1}}{c_k}, \quad \cdots \quad , \quad y_{n} \geq -\frac{y_{n+1}}{c_k}
\]
Then:
\begin{align*}
w_1 y_1 + w_2 y_2 + \cdots + w_n y_n + y_{n+1} 
&\geq -\frac{w_1 y_{n+1}}{c_k} - \frac{w_2 y_{n+1}}{c_k} - \cdots \\
&\quad - \frac{w_n y_{n+1}}{c_k} + y_{n+1} \\
&= \left(1 - \frac{w_1 + w_2 + \cdots + w_n}{c_k} \right) y_{n+1}
\end{align*}

Since $w_1 + w_2 + \cdots +  w_{n} \leq c_{k}$ by definition, the expression in parentheses is non-negative, so:

\[
\left(1 - \frac{w_1 + w_2 + w_3 + \cdots + w_n}{c_k} \right) y_{n+1} \geq 0
\]

Hence, $\mathbf{b}^\top y \geq 0$, and the second assertion is false.

\textbf{\underline{Case 2:}}  $y_{n+1} < 0$

To satisfy $A^\top y \geq 0$:
\[
c_1 y_1 + y_{n+1} \geq 0 \quad \Rightarrow \quad y_1 \geq -\frac{y_{n+1}}{c_1}
\]

Similarly,
\[
y_2 \geq -\frac{y_{n+1}}{c_1}, \quad y_3 \geq -\frac{y_{n+1}}{c_1}, \quad \cdots \quad , \quad y_{n} \geq -\frac{y_{n+1}}{c_1}
\]
Then:
\begin{align*}
w_1 y_1 + w_2 y_2 + \cdots + w_n y_n + y_{n+1} 
&\geq -\frac{w_1 y_{n+1}}{c_1} - \frac{w_2 y_{n+1}}{c_1} - \cdots \\
&\quad - \frac{w_n y_{n+1}}{c_1} + y_{n+1} \\
&= \left(1 - \frac{w_1 + w_2 + \cdots + w_n}{c_1} \right)y_{n+1}
\end{align*}

Since $c_1 \leq w_1 + w_2 + \cdots +  w_{n} $ by definition, the expression in parentheses is negative, so:

\[
\left(1 - \frac{w_1 + w_2 + w_3 + \cdots + w_n}{c_1} \right) y_{n+1} \geq 0
\]

Hence, $\mathbf{b}^\top y \geq 0$, and the second assertion is false.
\bigskip

Since the second assertion is false, then the first assertion must be true and the proof is complete.

\end{pf}

\printcredits

\section*{Acknowledgements}
Climate and environmental data were obtained from the Alberta Climate 
Information Service (ACIS), provided by Agriculture and Agri-food Canada 
(AGI) and its partners, Environment and Parks and Environment Canada 
\citep{ACIS2024}. Data available at \url{www.agriculture.alberta.ca/acis}.

\bibliographystyle{cas-model2-names}

\bibliography{mfgrefs}



\end{document}